\documentclass[11pt]{article}
\pdfoutput=1

\usepackage[margin=1.3in]{geometry}

\usepackage{graphicx}
\usepackage{amssymb,amsmath}
\usepackage{amsthm}
\usepackage{hyperref,xcolor}
\usepackage{mathrsfs}
\usepackage{extarrows}
\usepackage{textcomp}
\usepackage{enumerate}
\hypersetup{
    colorlinks,%
    citecolor=black,%
    filecolor=black,%
    linkcolor=black,%
    urlcolor=black
}
\makeatletter

\newdimen\bibspace
\renewenvironment{thebibliography}[1]{%
 \section*{\refname 
       \@mkboth{\MakeUppercase\refname}{\MakeUppercase\refname}}%
     \list{\@biblabel{\@arabic\c@enumiv}}%
          {\settowidth\labelwidth{\@biblabel{#1}}%
           \leftmargin\labelwidth
           \advance\leftmargin\labelsep
           \itemsep\bibspace
           \parsep\z@skip     %
           \@openbib@code
           \usecounter{enumiv}%
           \let\p@enumiv\@empty
           \renewcommand\theenumiv{\@arabic\c@enumiv}}%
     \sloppy\clubpenalty4000\widowpenalty4000%
     \sfcode`\.\@m}
    {\def\@noitemerr
      {\@latex@warning{Empty `thebibliography' environment}}%
     \endlist}

\makeatother

\numberwithin{equation}{section}

\newtheorem{theorem}{Theorem}[section]
\newtheorem{lemma}[theorem]{Lemma}
\newtheorem{proposition}[theorem]{Proposition}

\newtheorem{corollary}[theorem]{Corollary}
\newtheorem{remark}[theorem]{Remark}

\def\<{\langle}
\def\>{\rangle}

\def\bolangxian{\widetilde}

\def\hat{\widehat}

\def\d{\mathtt{d}}

\def\hat{\widehat}
\def\d{\mathtt{d}}
\def\bolangxian{\widetilde}

\begin{document}

\title{\textbf{Asymptotic behavior at infinity  of solutions of Lagrangian mean curvature equations
\footnote{Supported in part by National Natural Science Foundation of China (11871102 and 11631002).}
} \bigskip}

\author{Jiguang Bao\footnote{E-mail address: \textsf{jgbao@bnu.edu.cn}},\ \ Zixiao Liu\footnote{Corresponding author, E-mail address: \textsf{liuzixiao@mail.bnu.edu.cn}}\\
\small School of Mathematical Sciences, Beijing Normal University\\
\small Laboratory of Mathematics and Complex Systems, Ministry of Education\\
\small Beijing 100875, China
}

\maketitle

\begin{abstract}
We studied the asymptotic behavior of solutions with quadratic growth condition of a class of Lagrangian mean curvature equations $F_{\tau}(\lambda(D^2u))=f(x)$ in exterior domain, where $f$ satisfies a given asymptotic behavior at infinity. When $f(x)$ is a constant near infinity, it is not necessary to demand the quadratic growth condition anymore. These results are a kind of exterior Liouville theorem, and can also be regarded as an extension of theorems of Pogorelov \cite{Pogorelov}, Flanders \cite{Flanders} and Yuan \cite{Yu.Yuan1,Yu.Yuan2}.
 \\[1mm]
\end{abstract}

\setcounter{tocdepth}{1}
\tableofcontents

\section{Introduction}

In 2010, M.Warren \cite{Warren} first studied the minimal/maximal Lagrangian graph in $\left(\mathbb{R}^{n} \times \mathbb{R}^{n}, g_{\tau}\right)$, where
\begin{equation*}
g_{\tau}=\sin \tau \delta_{0}+\cos \tau g_{0}, \quad \tau \in\left[0, \frac{\pi}{2}\right],
\end{equation*}
is the linearly combined metric of standard Euclidean metric
\begin{equation*}
\delta_{0}=\sum_{i=1}^{n} d x_{i} \otimes d x_{i}+\sum_{j=1}^{n} d y_{j} \otimes d y_{j},
\end{equation*}
and the pseudo-Euclidean metric
\begin{equation*}
g_{0}=\sum_{i=1}^{n} d x_{i} \otimes d y_{i}+ \sum_{j=1}^{n} d y_{j} \otimes d x_{j}.
\end{equation*}
He showed that if $u\in C^2(\Omega),\ \Omega\subset\mathbb{R}^n$
is a solution of
\begin{equation}\label{equation}
F_{\tau}\left(\lambda\left(D^{2} u\right)\right)=C_0,\quad x\in\Omega,
\end{equation}
then the volume of
$(x,Du(x))$ is a maximal (for $\tau\in (0,\frac{\pi}{4})$) /minimal (for $\tau\in (\frac{\pi}{4},\frac{\pi}{2})$) among all homologous, $C^1$, space-like $n$-surfaces in $\left(\mathbb{R}^{n} \times \mathbb{R}^{n}, g_{\tau}\right)$,
where $C_0$ is a constant,  $\lambda(D^2u)=(\lambda_1,\lambda_2,\cdots,\lambda_n)$ are $n$ eigenvalues of the Hessian matrix $D^2u$ and  $$
F_{\tau}(\lambda):=\left\{
\begin{array}{ccc}
\displaystyle  \frac{1}{n} \sum_{i=1}^{n} \ln \lambda_{i}, & \tau=0,\\
\displaystyle  \frac{\sqrt{a^{2}+1}}{2 b} \sum_{i=1}^{n} \ln \frac{\lambda_{i}+a-b}{\lambda_{i}+a+b},
  & 0<\tau<\frac{\pi}{4},\\
  \displaystyle-\sqrt{2} \sum_{i=1}^{n} \frac{1}{1+\lambda_{i}}, & \tau=\frac{\pi}{4},\\
  \displaystyle\frac{\sqrt{a^{2}+1}}{b} \sum_{i=1}^{n} \arctan \displaystyle\frac{\lambda_{i}+a-b}{\lambda_{i}+a+b}, &
  \frac{\pi}{4}<\tau<\frac{\pi}{2},\\
  \displaystyle\sum_{i=1}^{n} \arctan \lambda_{i}, & \tau=\frac{\pi}{2},\\
\end{array}
\right.
$$
$a=\cot \tau, b=\sqrt{\left|\cot ^{2} \tau-1\right|}$.

If $\tau=0$, then (\ref{equation}) becomes the famous Monge-Amp\`ere equation $$
\det D^2u=e^{nC_0}.
$$
If $\tau=\frac{\pi}{2}$, then (\ref{equation}) becomes the special Lagrangian equation
\begin{equation*}
\sum_{i=1}^{n} \arctan \lambda_{i}\left(D^{2} u\right)=C_0.
\end{equation*}
In the same paper \cite{Warren}, M.Warren used change of variable and restated the Bernstein-type results of J\"orgens \cite{Jorgens}-Calabi \cite{Calabi}-Pogorelov \cite{Pogorelov}, Flanders \cite{Flanders}, and Yuan \cite{Yu.Yuan1,Yu.Yuan2} to give the following statement.
\begin{theorem}\label{Warren}
$u$ is a quadratic polynomial if $u \in C^{2}\left(\mathbb{R}^{n}\right)$ is a solution of (\ref{equation}) respectively, in the following cases, with
\begin{enumerate}
\item $D^2u>(-a+b)I$ for $\tau \in (0,\frac{\pi}{4})$;

\item $D^2u>0$ for $\tau=\frac{\pi}{4}$;

\item either
\begin{equation}\label{condition-temp-2}
    D^2u\geq -a I,
    \end{equation} or
   \begin{equation}\label{condition-temp-1}
    D^2u> -(a+b)I,\quad \left|\frac{bC_0}{\sqrt{a^2+1}}+\frac{n\pi}{4}\right|>\frac{n-2}{2}\pi,
    \end{equation}
for $\tau\in(\frac{\pi}{4},\frac{\pi}{2})$.
  \end{enumerate}
\end{theorem}

The aim of this paper is to study the asymptotic behavior of classical solutions of equation (\ref{equation}) in exterior domain, including the stability with respect to the right hand side term.  There are plenty of marvelous work on these type of questions.
The classical theorem by J\"orgens \cite{Jorgens}, Calabi \cite{Calabi} and Pogorelov \cite{Pogorelov} states that any convex classical solution of $\operatorname{det}D^{2} u=1$ on $\mathbb{R}^{n}$ must be a quadratic polynomial. See Cheng and Yau \cite{ChengandYau}, Caffarelli \cite{7} and Jost and Xin \cite{JostandXin} for different proofs and extensions.
For Monge-Amp\`ere equations with constant right hand side term in exterior domain, there are exterior J\"orgens-Calabi-Pogorelov type results \cite{FMM99} (for $n=2$) and \cite{CL}, which state that all convex solutions are asymptotic to quadratic polynomials (for $n=2$ we need additional $\log$-term) near infinity.  Pioneered by Guti\`errez and Huang \cite{EntireParabolicMA}, Bernstein-type result of  parabolic convex entire solution also holds for parabolic Monge-Amp\`ere type equation $-u_t\det D^2u=1$ . Many other type of equations including parabolic version with constant right hand side term in exterior domain is also studied through different strategies, see for instance \cite{BCGJ,WYSpecialLagrangian,Yu.Yuan1,Yu.Yuan2,ParabolicMA_constant1,
ParabolicMA_Constant2,parabolicMA_constant3,parabolicSigmaK_constant}. Also, these type of equations with perturbed constant right hand side or periodic right hand side are also studied a lot, see for instance \cite{BLZ,Peroidic_parabolicMA,Peroidic_MA,Peroidic_MA2}.

First, we consider the equation with constant right hand side in exterior domain
\begin{equation}\label{equation_ConstantRHS}
F_{\tau}(\lambda(D^2u))=C_0,\quad\text{in}\quad \mathbb{R}^n\setminus\overline{B_1}.
\end{equation}
We prove the following results that provide an asymptotic behavior at infinity. The idea is to use Legendre transform and apply the strategies in \cite{ExteriorLiouville}.

From now on, we let $n\geq 3$,  $I$ stand for the unit $n\times n$ matrix and let $\mathtt{Sym}(n)$ denote the set of symmetric $n\times n$ constant matrix.
For any $C,D\in\mathbb{R}, \tau\in [0,\frac{\pi}{2}]$, let$$
\mathcal{A}_{\tau}(C,D):=\{A\in\mathtt{Sym}(n):F_{\tau}(\lambda(A))=C,\text{ and }A\geq  DI\}.
$$
To shorten notation, if $\tau\in[0,\frac{\pi}{2}]$ is fixed, we  write $\mathcal{A}(C,D)$ instead of $\mathcal{A}_{\tau}(C,D)$.

\begin{theorem}\label{ConstantRHS1}

  Let $u \in C ^ { 4 } \left( \mathbb { R } ^ { n }\setminus\overline{B_1} \right)$ be a classical solution of
  (\ref{equation_ConstantRHS}) for some constant $C_0$ with $\tau\in(0,\frac{\pi}{4})$
  and satisfy \begin{equation}\label{temp-7}
  D^2u>(-a+b)I,\quad\text{in}\quad\mathbb{R}^n\setminus\overline{B_1}.
  \end{equation}
  Then  there exist $\gamma \in \mathbb { R } , \beta \in \mathbb { R } ^ { n } \text { and } A\in \mathcal {A}(C_0,-a+b)$ such that \begin{equation}
  \label{ConvergenceSpeed1}
  \limsup _ { | x | \rightarrow \infty } | x | ^ { n - 2+k} \left| D^k\left( u ( x ) - \left( \frac { 1 } { 2 } x ^ { \prime } A x + \beta \cdot x + \gamma \right) \right) \right| < \infty,\quad k=0,1,2.
  \end{equation}
\end{theorem}
\begin{theorem}\label{ConstantRHS3}
  Let $u \in C ^ { 4 } \left( \mathbb { R } ^ { n } \setminus\overline{B_1}\right)$ be a classical solution of (\ref{equation_ConstantRHS}) for some constant $C_0$ with $\tau=\frac{\pi}{4}$ and satisfy
  \begin{equation}\label{temp-11}
  D^2u>-I,\quad\text{in}\quad\mathbb{R}^n\setminus\overline{B_1},
  \end{equation}
  Then there exist  $\gamma \in \mathbb { R } , \beta \in \mathbb { R } ^ { n } \text { and } A\in \mathcal {A}(C_0,-1)$ such that (\ref{ConvergenceSpeed1}) holds.
\end{theorem}
\begin{theorem}\label{ConstantRHS2}
  Let $u \in C ^ { 4 } \left( \mathbb { R } ^ { n } \setminus\overline{B_1}\right)$ be a classical solution of (\ref{equation_ConstantRHS}) for some constant $C_0$ with $\tau\in(\frac{\pi}{4},\frac{\pi}{2})$ and
  satisfy  \begin{equation}\label{condition-temp-3}
  D^2u\geq -(a-\varepsilon_0)I,
  \end{equation}
  for some $\varepsilon_0>0$ or (\ref{condition-temp-1}).
  Then there exist $\gamma \in \mathbb { R } , \beta \in \mathbb { R } ^ { n } \text { and }$ $A\in\mathcal{A}(C_0,-a+\varepsilon_0)$ or $ A\in \mathcal {A}(C_0,-(a+b))$ respectively such that (\ref{ConvergenceSpeed1}) holds.
\end{theorem}

By comparison principle as in \cite{CL}, the global Theorem \ref{Warren} can be obtained by these exterior behavior results i.e. Theorems \ref{ConstantRHS1}, \ref{ConstantRHS3} and \ref{ConstantRHS2}, apart from the situation of $\tau\in(\frac{\pi}{4},\frac{\pi}{2})$ with only (\ref{condition-temp-2}) instead of (\ref{condition-temp-3}).

Second, we also consider  equation  with right hand side of a perturbation of suitable constant as the following
\begin{equation}\label{equation_NONConstant}
  F _ { \tau } \left( \lambda(D ^ { 2 } u) \right) =\frac{\sqrt{a^2+1}}{2b} \sum _ { i = 1 } ^ { n } \ln \frac { \lambda _ { i } + a - b } { \lambda _ { i } + a + b } = f ( x ) , \quad \text { in } \quad \mathbb { R } ^ { n },
\end{equation}
for $\tau\in(0,\frac{\pi}{4})$.
Especially in this situation, $a>b$ always holds.

The main results in this part include

\begin{theorem}\label{result1}
Let $u \in C ^ { 4 } \left( \mathbb { R } ^ { n } \right)$ be a classical solution of (\ref{equation_NONConstant}) and satisfy \begin{equation}\label{boundedHessian}
(-a+b)I<D^2u\leq MI,
\end{equation}
$\text{for some constant $M$.}$
Assume that $
  f \in C ^ { 2} \left( \mathbb { R } ^ { n } \right)
  $,and \begin{equation}\label{Low-Regular-Condition}
  \limsup _ { | x | \rightarrow \infty } | x | ^ { n + \varepsilon } | D^k( f ( x ) - f ( \infty ) ) | < \infty,\quad k=0,1,2,
  \end{equation}
  for some $\varepsilon>0$.
Then there exist $\gamma \in \mathbb { R } , \beta \in \mathbb { R } ^ { n } \text { and } A\in \mathcal {A}(f(\infty),-a+b)$ such that (\ref{ConvergenceSpeed1}) holds.
\end{theorem}

\begin{theorem}\label{result2}
  Let $u \in C ^ {m+2} \left( \mathbb { R } ^ { n } \right)$ be a classical solution of  (\ref{equation_NONConstant})
  and satisfy (\ref{boundedHessian}).
  Assume that $f\in C^m(\mathbb{R}^n)$, and \begin{equation}\label{High-Regular-Condition}
  \limsup_ { | x | \rightarrow \infty } | x | ^ { \zeta + k } \left| D ^ { k } ( f ( x ) - f(\infty) ) \right| < \infty ,\quad\forall k = 0,1 , \cdots , m,
  \end{equation}
  for some $\zeta>2,m\geq 3$.
  Then there exist $\gamma \in \mathbb { R } , \beta \in \mathbb { R } ^ { n }$ and $A \in \mathcal {A}(f(\infty),-a+b)$ such that \begin{equation}\label{ConvergenceSpeed2}
  \limsup _ { | x | \rightarrow \infty } | x | ^ {\min\{\zeta,n\} - 2 +k} \left| D^k\left( u ( x ) - \left( \frac { 1 } { 2 } x ^ { \prime } A x + \beta \cdot x + \gamma \right) \right) \right| < \infty,\ \ \forall k=0,1,\cdots,m+1.
  \end{equation}

\end{theorem}

\begin{remark}\label{example}

We can see that these two conditions on $f$
in Theorems \ref{result1} and  \ref{result2}
doesn't include each other. The condition (\ref{Low-Regular-Condition}) holds true for more ``un-regular'' $f$ than condition (\ref{High-Regular-Condition}) , which is originally proposed by Bao-Li-Zhang \cite{BLZ}. Condition (\ref{Low-Regular-Condition}) only demands the behavior of up to 2-order derivative of $f$, while condition (\ref{High-Regular-Condition}) demands the behavior of at least 3-order derivative of $f$.
Generally speaking, for $n\geq 5$, the condition (\ref{Low-Regular-Condition}) in Theorem \ref{result1} demands a higher convergence speed than condition (\ref{High-Regular-Condition}) in Theorem \ref{result2}, but the demand on regularity is the other way around.

On the one side, we  take $f(x)=|x|^{-(2+\varepsilon)}-1$ for $|x|\geq1,\ 1>\varepsilon>0$. Then it satisfies the condition (\ref{High-Regular-Condition})  in Theorem \ref{result2} by picking $\zeta:=2+\varepsilon>2$ but doesn't satisfies the condition (\ref{Low-Regular-Condition}) in Theorem \ref{result1}.

On the other side, we take $
  f(x):=e^{-m|x|}\sin(e^{|x|})-1,\ |x|\geq 1,\ m\geq 3.
  $ Then it satisfies the condition (\ref{Low-Regular-Condition}) in Theorem \ref{result1} because it has exponential decay up to $(m-1)$-order derivatives. But it doesn't satisfies the condition (\ref{High-Regular-Condition}) in Theorem \ref{result2} because its $m$-order derivative doesn't admits a limit at infinity.

\end{remark}

Next, we can weaken the assumption (\ref{boundedHessian}) in Theorems \ref{result1} and \ref{result2}
into the linear growth of gradient
\begin{equation}\label{Condition-LinearGrowth}
|Du(x)|\leq C(1+|x|),\quad\forall\ x\in\mathbb{R}^n,
\end{equation}
for some constant $C$. And further more, this linear growth of gradient condition can be weakened into the quadratic growth condition of $u$ itself
\begin{equation}\label{temp-9}
|u(x)|\leq C(1+|x|^2),\quad\forall x\in\mathbb{R}^n,
\end{equation} for some constant $C$. Since in the following Theorems \ref{lineargrowth_Result}, \ref{quadraticgrowth_Result} and Remark \ref{exteriorDomain}, we always assume a lower bound of Hessian to make the equation elliptic, we can weaken the additional condition (\ref{temp-9})
into \begin{equation}\label{Condition-QuadraticGrowth}
u(x)\leq C(1+|x|^2),\quad\forall x\in\mathbb{R}^n,
\end{equation} for some constant $C$.

To be more precise, we have the following stronger theorems than Theorems \ref{result1} and \ref{result2}.
\begin{theorem}[Linear Growth of Gradient]\label{lineargrowth_Result}
Let $u \in C ^ { 4 } \left( \mathbb { R } ^ { n } \right)$ be a classical solution of (\ref{equation_NONConstant}) with $D^2u>(-a+b)I$ and satisfy (\ref{Condition-LinearGrowth}).
Suppose that $f$ satisfies the assumptions as in Theorem \ref{result1} (resp. Theorem \ref{result2}), then we have the same results as in Theorem \ref{result1} (resp. Theorem \ref{result2}).
\end{theorem}
\begin{theorem}[Quadratic Growth of Solution]\label{quadraticgrowth_Result}  Let $u \in C ^ { 4 } \left( \mathbb { R } ^ { n } \right)$ be a classical solution of (\ref{equation_NONConstant}) with $D^2u>(-a+b)I$ and satisfy (\ref{Condition-QuadraticGrowth}).
Suppose that $f$ satisfies the assumptions as in Theorem \ref{result1} (resp. Theorem \ref{result2}) , then we have the same results as in Theorem \ref{result1} (resp. Theorem \ref{result2}).
\end{theorem}

As in the famous paper of Bao-Chen-Guan-Ji \cite{BCGJ}, such a quadratic growth condition is common and necessary for general k-Hessian equations and Hessian quotient equations. See for instance \cite{parabolicSigmaK_constant,2-Hessian,Nonpolynomial}. And as for the classical special Lagrangian equation, the linear growth condition of gradient is also studied by M.Warren and Y.Yuan in \cite{WYSpecialLagrangian}. Similar conditions are also considered in different paper, see for instance \cite{Quotient-quadratic-LRW,Quotient-quadratic-Complex,Quotient-quadratic-Double,Quotient-quadratic-what,Quotient-Hessian}.

\begin{remark}\label{exteriorDomain}

  By using extension theorems, all the results in Theorems \ref{result1}, \ref{result2}, \ref{lineargrowth_Result} and \ref{quadraticgrowth_Result} still hold when the equation (\ref{equation_NONConstant}) only holds in exterior domain i.e.
  \begin{equation}\label{equation_exteriorDomain}
  F_{\tau}\left(\lambda(D^{2} u)\right)=\frac{\sqrt{a^2+1}}{2b}\sum_{i=1}^{n} \ln \frac{\lambda_{i}+a-b}{\lambda_{i}+a+b}=f(x), \quad \text { in } \quad \mathbb{R}^{n}\setminus\overline{\Omega},
  \end{equation}
  where $\Omega$ is a bounded domain in $\mathbb{R}^n$.
\end{remark}

Apart from the asymptotic behavior of solutions of (\ref{equation_NONConstant}) we stated earlier, the method we adapt here also works for Monge-Amp\`ere equation  (i.e. $\tau=0$ situation) without demanding the boundedness of Hessian.
The Theorem \ref{ORI-BLZ} below is proved by Bao-Li-Zhang \cite{BLZ} and we obtain a new result in Theorem \ref{WorkonMongeAmpere} under a different setting.
\begin{theorem}\label{ORI-BLZ}Let $u \in C ^ { 0 } \left( \mathbb { R } ^ { n } \right)$ be a convex viscosity solution of
\begin{equation}\label{Monge-Ampere}
  \det D^2u=f(x),\quad\text{in}\quad\mathbb{R}^n,
  \end{equation}
  where  $
  f \in C ^ { 0 } \left( \mathbb { R } ^ { n } \right)
  $
  , $D^mf\ (m\geq 3)$ exist outside a compact subset of $\mathbb{R}^n$  and (\ref{High-Regular-Condition}) holds for some given $f(\infty)\in(0,+\infty)$.
  Then there exist $\gamma \in \mathbb { R } , \beta \in \mathbb { R } ^ { n }$ and $A\in \mathcal { A }_0(\frac{1}{n}\exp(f(\infty)),0)$ such that (\ref{ConvergenceSpeed2}) holds.
\end{theorem}
\begin{theorem}\label{WorkonMongeAmpere}Let $u \in C ^ { 0 } \left( \mathbb { R } ^ { n } \right)$ be a convex viscosity solution of (\ref{Monge-Ampere}),\ where $
  f \in C ^ { 0 } \left( \mathbb { R } ^ { n } \right)
  $
  , $D^2f$ exist outside a compact subset of $\mathbb{R}^n$  and there exists a $\varepsilon>0$ such that (\ref{Low-Regular-Condition}) holds for some given $f(\infty)\in(0,+\infty)$.
  Then there exist $\gamma \in \mathbb { R } , \beta \in \mathbb { R } ^ { n }$ and $A\in \mathcal {A}_0(\frac{1}{n}\exp(f(\infty)),0)$ such that (\ref{ConvergenceSpeed1}) holds.
\end{theorem}
\begin{remark}

The difference between the conditions on $f$ in Theorems \ref{ORI-BLZ} and \ref{WorkonMongeAmpere} is the same as in Remark \ref{example}.
As in Remark \ref{exteriorDomain}, all the results in Theorems \ref{ORI-BLZ} and  \ref{WorkonMongeAmpere} still hold when the equation (\ref{Monge-Ampere}) only happens on exterior domain i.e.
\begin{equation}\label{Monge-Ampere-ext}
\det D^2u=f(x),\quad\text{in}\quad\mathbb{R}^n\setminus\overline{\Omega},
\end{equation}
where $\Omega$ is a bounded domain in $\mathbb{R}^{n}$.
\end{remark}
For perturbed right hand side case of  $\tau=\frac{\pi}{4}$ , since the equation after Legendre transform becomes Laplacian operator, the proof is even simpler than the proof of Theorems \ref{result1} and \ref{result2}. To be more precise, we obtain the following  results.
\begin{theorem}\label{result3} Let $u\in C^4(\mathbb{R}^n)$ be a classical solution of the following equation with $\tau=\frac{\pi}{4}$,
\begin{equation}\label{equation_middletau}
  F_{\tau}(\lambda(D^2u))=-\sqrt{2}\sum_{i=1}^n\frac{1}{1+\lambda_i(D^2u)}=f(x),\quad\text{in}
  \quad\mathbb{R}^n,
\end{equation}
and satisfy \begin{equation}\label{boundedHessian_2}
-I<D^2u\leq MI,
\end{equation}
for some constant $M$. Assume that $f\in C^2(\mathbb{R}^n)$, and satisfy (\ref{Low-Regular-Condition}) for some $\varepsilon>0$. Then there exist $\gamma\in\mathbb{R},\ \beta\in\mathbb{R}^n$ and $A\in\mathcal{A}(f(\infty),-1)$ such that (\ref{ConvergenceSpeed1}) holds.
\end{theorem}
\begin{theorem}\label{result4}Let $u \in C^{m+2}\left(\mathbb{R}^{n}\right)$ be a classical solution of (\ref{equation_middletau}) with $\tau=\frac{\pi}{4}$ and satisfy (\ref{boundedHessian_2}) for some constant M. Assume that $f\in C^m(\mathbb{R}^n)$ and satisfies (\ref{High-Regular-Condition}) for some $\zeta>2,\ m\geq 3$. Then there exist $\gamma\in\mathbb{R},\ \beta\in\mathbb{R}^n$ and $A\in \mathcal{A}(f(\infty),-1)$ such that (\ref{ConvergenceSpeed2}) holds.
\end{theorem}
\begin{theorem}\label{lineargrowth_Result_2}
  Let $u\in C^4(\mathbb{R}^n)$ be a classical solution of (\ref{equation_middletau}) with $D^2u>-I$ and satisfy (\ref{Condition-LinearGrowth}). Suppose that $f$ satisfies the assumptions as in Theorem \ref{result3} (resp. Theorem \ref{result4}), then we have the same results as in Theorem \ref{result3} (resp. Theorem \ref{result4}).
\end{theorem}
\begin{theorem}\label{quadraticgrowth_Result_2}
  Let $u\in C^4(\mathbb{R}^n)$ be a classical solution of (\ref{equation_middletau}) with $D^2u>-I$ and satisfy (\ref{Condition-QuadraticGrowth}). Suppose that $f$ satisfies the assumptions as in Theorem \ref{result3} (resp. Theorem \ref{result4}), then we have the same results as in Theorem \ref{result3} (resp. Theorem \ref{result4}).
\end{theorem}

\begin{remark}
  As in Remarks \ref{example}  and   \ref{exteriorDomain}, similar discussion work for $\tau=\frac{\pi}{4}$ as well.
Whether similar result still holds for $\frac{\pi}{4}<\tau\leq \frac{\pi}{2}$ is still unclear.
\end{remark}

This paper is organized in the following order. In section
\ref{Section-ConstantSituation}, we prove the asymptotic behavior of solutions of equations (\ref{equation_ConstantRHS}) with constant right hand side.
Then, we move on to study the equation with perturbed right hand side. In section \ref{ConvergenceOfHessian} we prove that under suitable conditions, the Hessian $D^2u$ converge to suitable constant matrix $A\in\mathtt{Sym(n)}$ with $|x|^{-\alpha}$ convergence rate for some $\alpha>0$ (Theorem \ref{LimitofHessian}). In section \ref{AnalyzeLinearEquation} we prepare some necessary analysis of linearized elliptic equations that will be used later on. In sections \ref{ProofofResult1} and \ref{ProofOfResult2} we capture the linear and constant part of the solution at infinity, due to the difference of tools we apply, the conditions on RHS term $f(x)$ are slightly different from \cite{BLZ}. Finally in section \ref{InteriorEstimates} we prove that under suitable conditions of $f$, the bounded Hessian condition follows from interior Hessian estimate and linear growth condition of gradient, which follows from the line of
\cite{CompactnessMethod}. Also, as a corollary of Y.Y.Li's paper \cite{Estimate-YYLi}, we notice that the linear growth condition of gradient follows from interior gradient estimate and quadratic growth condition of solution itself. Eventually in section \ref{section-middle}, we provide the results for $\tau=\frac{\pi}{4}$.


\section{Constant Right Hand Side Situation}\label{Section-ConstantSituation}

In this section, we give an rigorous proof of Theorems \ref{ConstantRHS1}, \ref{ConstantRHS3} and \ref{ConstantRHS2} based on  Theorem 2.1 in \cite{ExteriorLiouville}.  For reading simplicity, we repeat the statement of this theorem.

\begin{theorem}[Exterior Liouville Theorem ]\label{ExteriorLiouville_C}
 Let $u$ be a smooth solution of
  \begin{equation*}
F\left(D^{2} u\right)=0,
\end{equation*}
in the exterior domain $\mathbb{R}^{n} \backslash \overline{B_1},$ $F$ is smooth, uniformly elliptic and  the level
set $\{M | F(M)=0\}$ is concave. Suppose for some constant $K$,
\begin{equation*}
\left\|D^{2} u\right\|_{L^{\infty}\left(\mathbb{R}^{n} \backslash \overline{B}_{1}\right)} \leq K.
\end{equation*}
Then there exist $\gamma\in\mathbb{R},\ \beta\in\mathbb{R}^n$ and $A\in\{A\in\mathtt{Sym}(n):F(A)=0,\ ||A||\leq K\}$ such that (\ref{ConvergenceSpeed1}) holds.
\end{theorem}

By the important characterise results on ellipticity and concavity by Caffarelli-Nirenberg-Spruck \cite{CNS}, $F(D^2u)=f(\lambda(D^2u))$ is uniformly elliptic and concave if
\begin{equation*}
\lambda \leq \frac{\partial f}{\partial \lambda_{i}} \leq \Lambda, \quad \forall i=1,2, \cdots, n
\end{equation*}
for some constants $0<\lambda\leq \Lambda<\infty$ and $f$ is a concave function.

Now we move on to study our targets. The section is organized in the following order. We will discuss the situation of $0<\tau<\frac{\pi}{4}$ situation in the first part, $\tau=\frac{\pi}{4}$ situation in the second part and $\frac{\pi}{4}<\tau<\frac{\pi}{2}$ in the third part.

\subsection{$0<\tau<\frac{\pi}{4}$ Situation}

In this part, we focus on equation (\ref{equation_ConstantRHS}) with $0<\tau<\frac{\pi}{4}$. We introduce  the following  Legendre transform and apply Theorem \ref{ExteriorLiouville_C} twice to prove the asymptotic behavior.

Let $u\in C^2(\mathbb{R}^n\setminus\overline{B_1})$ be a classical solution of (\ref{equation_ConstantRHS}) and satisfy (\ref{temp-7}) .
We can extend the function $u$ sufficiently smooth to $\mathbb{R}^n$
 such that $D^2u>-(a+b)I$ in $\mathbb{R}^n$.
Let $$
\overline{u}(x):=u(x)+\dfrac{a+b}{2}|x|^2,
$$
then by (\ref{temp-7}) we have
\begin{equation*}
D^{2} \bar{u}=D^{2} u+(a+b) I>2bI,\text{ in }\mathbb{R}^n.
\end{equation*}
Let $(\widetilde{x},\bolangxian{v})$ be the Legendre transform of $(x,\bar{u})$ i.e.
$$
\left\{
\begin{array}{ccc}
  \bolangxian{x}:=D\overline{u}(x)=Du(x)+(a+b)x\\
  D_{\bolangxian{x}}v(\bolangxian{x}):=x\\
\end{array}
\right.,
$$
and  we have
\begin{equation*}
D_{\widetilde{x}}^{2} v(\widetilde{x})=\left(D^{2} \bar{u}(x)\right)^{-1}=(D^2u(x)+(a+b)I)^{-1}<\frac{1}{2b}I.
\end{equation*}
This transform is exactly the classical Legendre transform, hence it is proved rigourously that such a scaler function $v(\bolangxian{x})$ exists.
Then we take \begin{equation}\label{LegendreTransform}
\bolangxian{u}(\bolangxian{x}):=\dfrac{1}{2}|\bolangxian{x}|^2-2b\cdot v(\bolangxian{x}),
\end{equation}
which satisfies
\begin{equation}\label{property-Legendre}
\widetilde{\lambda}_{i}\left(D^{2} \widetilde{u}\right)=1-2 b \cdot \frac{1}{\lambda_{i}+a+b}=\frac{\lambda_{i}+a-b}{\lambda_{i}+a+b}\in (0,1).
\end{equation}

Then we see that $\bolangxian{u}(\bolangxian{x})$ satisfies the following Monge-Amp\`ere type equation
\begin{equation}\label{temp-16}
G(\bolangxian{\lambda}(D^2\bolangxian{u})):=\sum _ { i = 1 } ^ { n } \ln \bolangxian{\lambda_i}=\frac{2b}{\sqrt{a^2+1}}C_0,\ \text{in }\mathbb{R}^n\setminus\overline{\Omega},
\end{equation}
for some bounded set $\Omega=D\overline{u}(B_1)=Du(B_1)+(a+b)B_1$.

Moreover, we see that for any $x\in\mathbb{R}^n,\ \widetilde{x}=D\bar{u}(x)$, $$
|\bolangxian{x}-\widetilde{0}|=|Du(x)-Du(0)+(a+b)x|
>2b|x|.
$$
Hence triangle inequality gives us \begin{equation}\label{limitofX}
|\widetilde{x}|\geq
-|\widetilde{0}|+|\widetilde{x}-\widetilde{0}|
> -|\widetilde{0}|+2b|x|.
\end{equation}
Especially  we have $$
\lim_{|x|\rightarrow\infty}|\widetilde{x}|=\infty.
$$
Now we are ready to prove Theorem \ref{ConstantRHS1}.
\begin{proof}[Proof of Theorem \ref{ConstantRHS1}]
  First we prove that there exists a positive lower bound such that $\bolangxian{\lambda_i}$ is uniformly strictly positive.

  It follows from equation (\ref{equation_ConstantRHS}) and  the fact that $\ln \dfrac{\lambda_i+a-b}{\lambda_i+a+b}\leq0,\quad\forall\ i=1,2,\cdots,n$, we have $$\ln\bolangxian{\lambda_i}=
  \ln \dfrac{\lambda_i+a-b}{\lambda_i+a+b}\geq \frac{2b}{\sqrt{a^2+1}}C_0,\quad\forall i=1,2,\cdots,n.
  $$
  Thus it follows immediately that $\bolangxian{\lambda_i}\geq e^{\frac{2b}{\sqrt{a^2+1}}C_0}>0$ and by $C_0<0$, \begin{equation}\label{temp_lowerbound_1}
  \lambda_i\geq \dfrac{2b}{1-e^{\frac{2b}{\sqrt{a^2+1}}C_0}}-a-b=-a+b+\dfrac{e^{\frac{2b}{\sqrt{a^2+1}}
  C_0}}{1-e^{\frac{2b}{\sqrt{a^2+1}}C_0}}2b=:-a+b+\delta.
  \end{equation}

Note that for $\widetilde{\lambda_i}\in [e^{\frac{2b}{\sqrt{a^2+1}}C_0},1]$, taking derivatives and we see $\frac{\partial G}{\partial \widetilde{\lambda}_{i}}(\widetilde{\lambda})$ has positive lower and upper bound, $\frac{\partial^{2} G}{\partial \tilde{\lambda}_{i}^{2}}(\tilde{\lambda})$ has negative upper bound.
Hence the  equation (\ref{temp-16}) is uniformly elliptic and concave.

  Applying Theorem \ref{ExteriorLiouville_C} we learn  that
  there exists some $$\widetilde{A}\in\{\widetilde{A}\in\mathtt{Sym}(n):G(\lambda(\widetilde{A}))=
  \frac{2b}{\sqrt{a^2+1}}C_0,\ \lambda_i(\widetilde{A})\in [e^{\frac{2b}{\sqrt{a^2+1}}C_0},1]\},$$ such that
  \begin{equation}\label{LimitofHessian-AfterLegendre}
  \lim_{|\bolangxian{x}|\rightarrow+\infty}D^2\bolangxian{u}(\bolangxian{x})= \widetilde{A}.
  \end{equation}
  From the relationship of Legendre transform, we have $$
  D^2\bolangxian{u}(\bolangxian{x})=I-2b(D^2u(x)+(a+b)I)^{-1},\quad\text{and}\quad
  \bolangxian{x}=Du(x)+(a+b)x.
  $$

  Now we show that all the eigenvalues  $\lambda_i( \widetilde{ A})$  are strictly less than 1, which tells us $I-\widetilde{A}$ is an invertible matrix.  Argue by contradiction, by rotating the $\widetilde{x}$ -space to make $\widetilde{A}$ diagonal, we may assume that $\widetilde{A}_{11}=$
  $1.$ Then by the asymptotic behavior of $D\widetilde{u}$ and hence $Dv$ (take $k=1$ in formula (\ref{ConvergenceSpeed1}) of Theorem \ref{ExteriorLiouville_C}), we  have the following asymptote for some $\widetilde{\beta}_1\in\mathbb{R}$
  \begin{equation*} D_{\widetilde{x}_{1}}v=\widetilde{\beta}_1+O(|\widetilde{x}|^{1-n}),
\end{equation*}
 as $|\widetilde{x}|\rightarrow\infty$. Thus we infer from the definition of Legendre transform (\ref{LegendreTransform}) and (\ref{limitofX}) that
\begin{equation}\label{strip-argument}
x_1=D_{\widetilde{x}_1}v(\widetilde{x})=\widetilde{\beta}_{1}+O\left(|\widetilde{x}|^{1-n}\right),
\end{equation}
as $|\widetilde{x}|\rightarrow\infty$.
which means that $\mathbb{R}^{n} \backslash \overline{B_1}$ is bounded in the $x_{1}$ -direction, hence this becomes a contradiction.  Thus $\lambda_i(\widetilde{A})<1$ strictly for every $i=1,2,\cdots,n$. Hereinafter we will state similar argument as ``strip argument'' for short.

 Together with the fact (\ref{limitofX}) and (\ref{LimitofHessian-AfterLegendre}), we see that
 $$
  \lim_{|x|\rightarrow+\infty}D^2u(x)
  =\dfrac{1}{2b}(I-\bolangxian{A})^{-1}-(a+b)I,
  $$
  which is a bounded matrix. This limit together with the regularity assumption of extended function $u\in C^4(\mathbb{R}^n)$ tells us $$
  D^2u(x)\leq M,\quad\forall x\in\mathbb{R}^n.
  $$

 Applying formula (\ref{temp_lowerbound_1}), we see that   for $\lambda_i\in [-a+b+\delta,M]$, $$
  \dfrac{\partial F_{\tau}}{\partial \lambda_i}(\lambda)=\frac{\lambda_{i}+a+b}{\lambda_{i}+a-b} \cdot \frac{2 b}{\left(\lambda_{i}+a+b\right)^{2}}=\frac{2 b}{\left(\lambda_{i}+a\right)^{2}-b^{2}}\in [\frac{2b}{(M+a)^2-b^2},\frac{2b}{(b+\delta)^2-b^2}],
  $$
  and $$
  \dfrac{\partial^2 F_{\tau}}{\partial \lambda_i^2}(\lambda)=-\frac{4 b\left(\lambda_{i}+a\right)}{\left[\left(\lambda_{i}+a\right)^{2}-b^{2}\right]^{2}}<0.
  $$

  Thus  $F_{\tau}$ is uniformly elliptic and concave, then Theorem \ref{ExteriorLiouville_C} gives us the result.
\end{proof}

\subsection{$\tau=\frac{\pi}{4}$ Situation}

Let $u\in C^4(\mathbb{R}^n\setminus\overline{B_1})$ be a classical solution of (\ref{equation_ConstantRHS}) with $\tau=\frac{\pi}{4}$ and satisfy (\ref{temp-11}).
Let $$
\overline{U}(x):=u(x)+\dfrac{1}{2}|x|^2.
$$
We can extend the function $u$ sufficiently smooth to $\mathbb{R}^n$ such that $D^2u>-I$ in $\mathbb{R}^n$ and hence $D^2\overline{U}>0$ in $\mathbb{R}^n$.

Let $(\widetilde{x},V)$ be the Legendre transform of $(x,\overline{U})$ i.e.
\begin{equation}\label{LegendreTransform2}
\left\{
\begin{array}{ccc}
  \bolangxian{x}:=D\overline{U}(x)=Du(x)+x\\
  D_{\bolangxian{x}}V(\bolangxian{x}):=x\\
\end{array}
\right.,
\end{equation}
and we have
\begin{equation*}
D_{\bolangxian{x}}^2V(\bolangxian{x})=(D^2\overline{U}(x))^{-1}.
\end{equation*}
Then we see that $V(\bolangxian{x})$ satisfies: \begin{equation}\label{equation_V}
-\Delta_{\bolangxian{x}} V=\dfrac{\sqrt{2}C_0}{2},\quad\text{in}\quad\mathbb{R}^n\setminus\overline{\Omega},
\end{equation}
where ${\Omega}=D\overline{U}(B_1)=Du(B_1)+B_1$.

\begin{proof}[Proof of Theorem \ref{ConstantRHS3}.]
  By the Legendre transform (\ref{LegendreTransform2}), the original equation is changed into formula (\ref{equation_V}). Since $D^2u>-I$, hence from the equation (\ref{equation_ConstantRHS}) we obtain that for all $i=1,2,\cdots,n$ $$
  \dfrac{\sqrt{2}}{\lambda_i(D^2u)+1}
  =\sum_{j\not=i}\dfrac{-\sqrt{2}}{\lambda_j(D^2u)+1}-C_0< -C_0.
  $$
  Hence \begin{equation}\label{formula}
  D^2_{\bolangxian{x}}V(\bolangxian{x})
  =(D^2u(x)+I)^{-1}\leq \dfrac{-C_0}{\sqrt{2}}I.
  \end{equation}
  Thus $D^2_{\bolangxian{x}}V(\bolangxian{x})$ is positive and bounded, applying Theorem \ref{ExteriorLiouville_C} and we learned that the limit of $D^2_{\bolangxian{x}}V(\bolangxian{x})$ exists as $|\bolangxian{x}|\rightarrow\infty$.

  As the argument in the proof of Theorem \ref{ConstantRHS1}, strip argument (\ref{strip-argument}) tells us the limit of $D^2u(x)$ as $|x|\rightarrow\infty$ also exists, hence $D^2u$ is bounded from above. Hence $F_{\tau}$ is uniformly elliptic and concave with respect to the set of solutions and Theorem \ref{ExteriorLiouville_C} gives us the result.
\end{proof}

\subsection{$\frac{\pi}{4}<\tau<\frac{\pi}{2}$ Situation}

In this part, we focus on equation (\ref{equation_ConstantRHS}) with $\frac{\pi}{4}<\tau<\frac{\pi}{2}$.
We  prove Theorem \ref{ConstantRHS2} by using the following important identity formula, which is proved by taking derivatives. (See for example \cite{Warren,huang2019entire}).
  When $\lambda_i>-a-b,\ \forall i=1,2,\cdots,n$, then
  \begin{equation}\label{identity-inuse}
  \sum_{i=1}^n\arctan \frac{\lambda_{i}+a-b}{\lambda_{i}+a+b}=\sum_{i=1}^n\arctan \left(\frac{\lambda_{i}+a}{b}\right)-\frac{n\pi}{4}.
  \end{equation}

First, we introduce the following result on asymptotic behavior of classical solutions of special Lagrangian equation (\ref{equation}) with $\tau=\frac{\pi}{2}$. See Theorem 1.1 of \cite{ExteriorLiouville}.
\begin{theorem}\label{Yu.Yuan_Result}

  Let $u$ be a smooth solution of (\ref{equation_ConstantRHS}) with $\tau=\frac{\pi}{2}$, and we assume that
 $|C_0|>(n-2) \pi / 2$. Then there exist $\gamma\in\mathbb{R},\ \beta\in\mathbb{R}^n$ and $A\in \mathcal{A}_{\tau}(C_0,-\infty)$ such that (\ref{ConvergenceSpeed1}) holds.
\end{theorem}

Applying Theorem \ref{Yu.Yuan_Result} and the identity formula (\ref{identity-inuse}), we can immediately obtain the second part of Theorem \ref{ConstantRHS2}.
\begin{proposition}\label{Result-Part1}
  Let $u \in C ^ { 4 } \left( \mathbb { R } ^ { n }\setminus\overline{B_1} \right)$ be a classical solution of (\ref{equation_ConstantRHS}) with $\tau\in(\frac{\pi}{4},\frac{\pi}{2})$ and satisfy condition (\ref{condition-temp-1}).
  Then there exist $\gamma \in \mathbb { R } , \beta \in \mathbb { R } ^ { n } \text { and } A\in \mathcal {A}(C_0,-(a+b))$ such that (\ref{ConvergenceSpeed1}) holds.
\end{proposition}
\begin{proof}[Proof.]
  Setting
  \begin{equation}\label{temp-21}
v(x)=\frac{u(x)}{b}+\frac{a}{2 b}|x|^{2},
\end{equation}
since $u\in C^4$ satisfies equation (\ref{equation_ConstantRHS}) and condition (\ref{condition-temp-1}), by (\ref{identity-inuse}) we obtain
\begin{equation*}
\sum_{i=1}^{n} \arctan \frac{\lambda_{i}(D^2u)+a}{b}=\frac{b}{\sqrt{a^2+1}}(C_0+\frac{n}{4} \frac{\sqrt{a^{2}+1}}{b} \pi),
\end{equation*}
i.e.
\begin{equation}\label{temp-10}
\sum_{i=1}^n\arctan \lambda_i(D^2v)=\frac{b}{\sqrt{a^2+1}}C_0
+\frac{n}{4}\pi.
\end{equation}
Condition (\ref{condition-temp-1}) together with
Theorem \ref{Yu.Yuan_Result} tells us there exist $\overline{\gamma}\in\mathbb{R},\ \overline{\beta}\in\mathbb{R}^n,$ and $\overline{A}\in\mathcal{A}_{\frac{\pi}{2}}(\frac{b}{\sqrt{a^{2}+1}}C_{0}
+\frac{n}{4}\pi,
-\infty)$ such that (\ref{ConvergenceSpeed1}) holds.

Hence the result follows from the definition of $v(x)$.
\end{proof}

Moreover, if we assume that $D^2u\geq -(a-\varepsilon_0)I$ for some $\varepsilon_0>0$, then we can similarly obtain a corresponding asymptotic behavior result by Theorem \ref{ExteriorLiouville_C} here. The strategy here doesn't work for $D^2u\geq -aI$ situation.

\begin{theorem}\label{Yu.Yuan_Corollary}Let $u\in C^4(\mathbb{R}^n\setminus\overline{B_1})$ be a classical solution of (\ref{equation_ConstantRHS}) with $\tau=\frac{\pi}{2}$ and
satisfy $D^2u\geq \varepsilon_0I$. Then there exist $\gamma\in\mathbb{R}, \beta\in\mathbb{R}^n$ and $A\in\mathcal{A}_{\frac{\pi}{2}}(C_0,\varepsilon_0)$ such that (\ref{ConvergenceSpeed1}) holds.

\end{theorem}

\begin{proof}[Proof.]

Extend $u$ sufficiently smooth to $\mathbb{R}^n$ such that $D^2u>\frac{\varepsilon_0}{2}I$ in $\mathbb{R}^n$.
Let $(\widetilde{x},\widetilde{u})$ be the Legendre transform of $(x,u)$ i.e. $$
\left\{
\begin{array}{ccc}
  \widetilde{x}:=Du(x)\\
  D_{\widetilde{x}}\widetilde{u}(\widetilde{x}):=x\\
\end{array}
\right.,
$$
and we have \begin{equation*}
D_{\bar{x}}^{2} \widetilde{u}(\widetilde{x})=\left(D^{2} {u}(x)\right)^{-1}.
\end{equation*}
Hence  $$
\bolangxian{\lambda}_i(D^2\bolangxian{u})=\frac{1}{\lambda_i}\in(0,\frac{2}{\varepsilon_0}).
$$
Also, due to $D^2u>\frac{\varepsilon_0}{2}I$ in $\mathbb{R}^n,$ we also have the following relationship for all $x^1, x^2\in\mathbb{R}^n$ \begin{equation}\label{temp-8}
\begin{array}{llll}
|\widetilde{x}^2-\widetilde{x}^1|^2& =& |Du(x^2)-Du(x^1)|^2\\
&\geq & \frac{\varepsilon_0^2}{4}\left|x^{2}-x^{1}\right|^{2}.\\
\end{array}
\end{equation}
Thus $\widetilde{u}$ satisfies  $$
G(\widetilde{\lambda}_i):=\sum_{i=1}^n\arctan \left(\frac{1}{\widetilde{\lambda}_i}
\right)=C_0,\quad\text{in}\quad \mathbb{R}^n\setminus\overline{\Omega},
$$
where $\Omega=Du(B_1)$ is a bounded domain.
By taking derivatives,  we see that
\begin{equation*}
\frac{\partial}{\partial \widetilde{\lambda}_{i}}G(\widetilde{\lambda}_i)=
\frac{1}{1+(\frac{1}{\widetilde{\lambda}_i})^2}\cdot (-\frac{1}{\widetilde{\lambda}_i^2})=-\frac{1}{1+\widetilde{\lambda}_i^2}\in(-1,-\frac{\varepsilon_0^2}{4+\varepsilon_0^2}),
\end{equation*}
\begin{equation*}
\frac{\partial^{2}}{\partial \widetilde{\lambda}_{i}^{2}}\left(
G(\widetilde{\lambda}_i)
\right)=\frac{2 \widetilde{\lambda}_{i}}{\left(1+\widetilde{\lambda}_{i}^{2}\right)^{2}}>0.
\end{equation*}
Hence the equation satisfied by $\widetilde{u}$ is uniformly elliptic and convex,
applying Theorem \ref{ExteriorLiouville_C}  we see that there exists a $
\widetilde{A} \in \operatorname{Sym}(n) $ satisfying $\sum_{i=1}^{n} \arctan \left(\frac{1}{\lambda_i(\widetilde{A})}\right)=C_0$
such that $$
\lim _{|\widetilde{x}| \rightarrow+\infty} D^{2} \widetilde{u}(\widetilde{x})=\widetilde{A}.
$$
By strip argument as (\ref{strip-argument}), we see that $\widetilde{A}$ is invertible.
Hence from the definition of Legendre transform and  (\ref{temp-8}) tells us $|\tilde{x}|\rightarrow\infty, $ as $|x|\rightarrow\infty$, we see that the limit of $D^2u(x)$ as $|x|\rightarrow\infty$ exists as well. Hence $D^2u$ is bounded on $\mathbb{R}^n$.
Now we verify that the equation (\ref{equation_ConstantRHS}) with $\tau=\frac{\pi}{2}$ under these conditions  is also uniformly elliptic and concave. By taking derivatives, we have $$
  \frac{\partial}{\partial\lambda_i}(\sum_{i=1}^n\arctan\lambda_i)=\frac{1}{1+\lambda_i^2},
  $$
  $$
  \frac{\partial^2}{\partial\lambda_i^2}(\sum_{i=1}^n\arctan\lambda_i)
  =-\frac{2\lambda_i}{(1+\lambda_i^2)^2}<0.
  $$
  Due to $D^2u$ is bounded and non-negative, hence the operator is uniformly elliptic and  concave.

  Apply Theorem \ref{ExteriorLiouville_C} and the result follows immediately.
\end{proof}
As in the proof of Proposition \ref{Result-Part1}, we can similarly obtain the following result by using identity formula (\ref{identity-inuse}).

\begin{proposition}\label{Result-Part2} Let $u \in C ^ { 4 } \left( \mathbb { R } ^ { n }\setminus\overline{B_1} \right)$ be a classical solution of (\ref{equation_ConstantRHS}) with $\tau\in(\frac{\pi}{4},\frac{\pi}{2})$,
  satisfying condition (\ref{condition-temp-3}).
  Then there exist $\gamma \in \mathbb { R } , \beta \in \mathbb { R } ^ { n } \text { and } A\in \mathcal {A}(C_0,-a+\varepsilon_0)$ such that (\ref{ConvergenceSpeed1}) holds.
\end{proposition}

\begin{proof}[Proof.]
Let $v$ be the function defined as in (\ref{temp-21}),
then formula (\ref{identity-inuse}) tells us $v$ satisfies equation (\ref{temp-10}). Since $D^2u\geq -(a-\varepsilon_0)I$, we have  $D^2v\geq \varepsilon_0I$.

Apply Theorem \ref{Yu.Yuan_Corollary}, the result follows immediately.
\end{proof}

Theorem \ref{ConstantRHS2} is the combination of Proposition \ref{Result-Part1} and Proposition \ref{Result-Part2}.

\section{Convergence of Hessian at infinity}\label{ConvergenceOfHessian}

In this section, we    study the asymptotic behavior at infinity of Hessian matrix of classical solution of (\ref{equation_NONConstant}) with a fixed $\tau\in(0,\frac{\pi}{4})$.

Some basic and important results on Monge-Amp\`ere equation (\ref{Monge-Ampere}) are important in our proof.
The following lemma can be found in the proof of Theorem 1.2 in \cite{BLZ}, which is based on the argument by Caffarelli-Li \cite{CL}.
In the following results on Monge-Amp\`ere equation (\ref{Monge-Ampere}) i.e. Lemma \ref{roughestimate}, Corollary \ref{corollary_estimate}, Theorem \ref{corollary_estimate3} and Remark \ref{corollary_estimate2},
we  assume $f(\infty)=1$ for simplicity.

\begin{lemma}\label{roughestimate}
  Let $u \in C ^ { 0 } \left( \mathbb { R } ^ { n } \right)$  be a convex viscosity solution of (\ref{Monge-Ampere}) with $u(0)=\min_{\mathbb{R}^n}u=0$  , where  $
  0<f \in C ^ { 0 } \left( \mathbb { R } ^ { n } \right)
  $
  and $$f^{\frac{1}{n}}-1\in L^n(\mathbb{R}^n).
$$
  Then there exists a linear transform $T$ satisfying $\det T=1$ such that $v:=u\circ T$ satisfies \begin{equation}\label{roughestimate_appendix}
  \left|v-\dfrac{1}{2}|x|^2\right|\leq C_1|x|^{2-\varepsilon},\quad \forall |x|\geq R_0.
  \end{equation}
  for some $C_1>0,\ \varepsilon>0$ and $R_0\gg 1$.
\end{lemma}

\begin{corollary}\label{corollary_estimate}
  Let $u$ satisfy the same conditions as in Lemma \ref{roughestimate} , $
  0<f \in C ^ { 0 } \left( \mathbb { R } ^ { n } \right) $
  and $\exists \beta > 1 $ such that
  $$
  \limsup_{|x|\rightarrow\infty}|x|^{\beta}|f(x)-1|<\infty.
  $$
  Then the same result in Lemma \ref{roughestimate} holds.
\end{corollary}

\begin{proof}[Proof.]Note that for $a,b$ near 1 (bounded and away from origin), we have the following inequality $$
a^{\frac{1}{n}}-b^{\frac{1}{n}}\leq C\cdot (a-b),
$$
due to the derivative of $x^{\frac{1}{n}}$ at $x=1$ is $\dfrac{1}{n}<1$. Hence we have $$
\left( \int _ { \mathbb{R}^n\setminus B_{1} } \left| f ( z ) ^ { \frac { 1 } { n } } - 1 \right| ^ { n } d z \right) ^ { \frac { 1 } { n } }\leq C\cdot
\left( \int _ { \mathbb{R}^n\setminus B_{1} } \left| f ( z ) - 1 \right| ^ { n } d z \right) ^ { \frac { 1 } { n } },
$$
where $C=C(\inf f,\sup f,n)$. Hence we have $$
LHS\leq C\cdot \left( \int _ { \mathbb { R } ^ { n } \backslash B _ { 1 } } |\dfrac{1}{|x|^{\beta}}| ^ { n } d z \right) ^ { \frac { 1 } { n } }<\infty\ \quad (\beta>1).
$$
Then the conditions of Lemma \ref{roughestimate} are met and the results holds.
\end{proof}

 As a consequence, we have the following result on Monge-Amp\`ere equation. The proof is similar to the one in Bao-Li-Zhang \cite{BLZ} and in Caffarelli-Li \cite{CL}. Since there are some differences from their proof, we provide the details here for reading simplicity.

\begin{theorem}\label{corollary_estimate3}
  Let $u \in C ^ { 0 } \left( \mathbb { R } ^ { n } \right)$ be a convex viscosity solution of (\ref{Monge-Ampere}),  $
  f \in C ^ {\alpha} \left( \mathbb { R } ^ { n } \right) $, $\alpha\in(0,1) , $ and satisfy
  \begin{equation}\label{0-orderCondition}
\limsup_ { | x | \rightarrow \infty } | x | ^ { \beta } | f ( x ) - 1| < \infty,
  \end{equation}
  \begin{equation}\label{Holder-Condition}
  [f]_{C^{\alpha}(B_{\frac{3}{2}R}\setminus B_{\frac{R}{2}})}\cdot R^{\alpha+\gamma}\leq C,\quad\forall R>1,
  \end{equation}
   for some $\beta>1$, $\gamma>0$.
  Then there exist $A\in\mathcal{A}_0(0,0)$, $ C \left( n , R _ { 0 } , \alpha, \beta , \gamma, c _ { 0 }\right),$ such that \begin{equation}\label{ConvergeRateofHessian}
  |D^2u(x)-A|\leq\dfrac{C}{|x|^{\min\{\beta,\varepsilon,\gamma\}}},\ \ \forall |x|\geq R_1,
  \end{equation}
  and
  \begin{equation}\label{HolderRegularity_MA}
    ||D^2u||_{C^{\alpha}(\mathbb{R}^n)}\leq C,
  \end{equation}
  where  $\varepsilon$, and $ R_0$ are positive constants come from Corollary \ref{corollary_estimate} (or say, Lemma \ref{roughestimate}) and  $R _ { 1 }  > R _ { 0 }$ depends only on $n , R _ { 0 } , \varepsilon , \beta , c _ { 0 }$.
\end{theorem}

\begin{proof}[Proof]

  By Corollary \ref{corollary_estimate}, we see that there exist a linear transform $T$ and $\varepsilon>0,\ C_1>0,\ R_0\gg 1$ such that $v:=u\circ T$ satisfies (\ref{roughestimate_appendix})

  Now in order to obtain pointwise decay speed estimate at infinity, we focus on sufficiently far ball, scale back to unit size, and then apply interior regular estimate.
  Set $$
  w(x):=v(x)-\dfrac{1}{2}|x|^2.
  $$
  For $|x|=R>2 R_{0},$ let
\begin{equation*}
v_{R}(y)=\left(\frac{4}{R}\right)^{2} v\left(x+\frac{R}{4} y\right), \quad\text{and}\quad w_{R}(y)=\left(\frac{4}{R}\right)^{2} w\left(x+\frac{R}{4} y\right) \quad|y| \leq 2.
\end{equation*}

By (\ref{roughestimate_appendix}), we have \begin{equation*}
\left\|v_{R}\right\|_{L^{\infty}\left(B_{2}\right)} \leq C, \quad\left\|w_{R}\right\|_{L^{\infty}\left(B_{2}\right)} \leq C R^{-\varepsilon},
\end{equation*}
for some constants C uniform to $R\geq R_0$ and
\begin{equation*}
v_{R}(y)-\left(\frac{1}{2}|y|^{2}+\frac{4}{R} x \cdot y+\frac{8}{R^{2}}|x|^{2}\right)=O\left(R^{-\varepsilon}\right), \quad \forall y\in B_{2}.
\end{equation*}
In order to apply interior estimate, we set $$\overline{v}_{R}(y):=v_{R}(y)-\frac{4}{R} x \cdot y-\frac{8}{R^{2}}|x|^{2}.$$
If $R>R_{1}$ with $R_{1}$ sufficiently
large, the set
\begin{equation*}
\Omega_{R}:=\left\{y \in B_{2} ; \quad \overline{v}_{R}(y) <1\right\}
\end{equation*}
is between $B_{1.2}$ and $B_{1.8} .$ The equation satisfied by $\overline{v}_{R}$ is \begin{equation*}
\operatorname{det}\left(D^{2} \overline{v}_{R}(y)\right)=f\left(x+\frac{R}{4} y\right)=:f_{ R}(y), \quad \text { in } B_{2}.
\end{equation*}

By definition of $f_{R}$, the assumptions (\ref{0-orderCondition}) and (\ref{Holder-Condition}), we have
$$
||f_{R}-1||_{C^{0}(\overline{B_2})}\leq C R^{-\beta},\quad\forall R>1.
$$
$$
\begin{array}{llll}
[f_{R}]_{C^{\alpha}(\overline{B_2})}
&=&\displaystyle\max_{y_1,y_2\in B_2}
\dfrac{|f(x+\frac{R}{4}y_1)-f(x+\frac{R}{4}y_2)|}{|y_1-y_2|^{\alpha}}\\&
=& \displaystyle \max_{z_1,z_2\in B_{\frac{3}{2}R}(x)\setminus B_{\frac{R}{2}}(x)}
\dfrac{|f(z_1)-f(z_2)|}{|z_1-z_2|^{\alpha}}
\cdot (\frac{R}{4})^{\alpha}\\
&\leq& CR^{-\gamma}.
\end{array}
$$
Combine these two parts we have
$$
||f_{R}-1||_{C^{\alpha}(\overline{B_2})}\leq CR^{-\min\{\beta,\gamma\}}.
$$
Applying the interior estimate by Caffarelli \cite{7}, Jian and Wang \cite{25} on $\Omega_{1,v}$, we have
\begin{equation}\label{CalphaEstimate}
\left\|D^{2} v_{R}\right\|_{C^{\alpha}\left(B_{1,1}\right)}=\left\|D^{2} \overline{v}_{R}\right\|_{C^{\alpha}\left(B_{1.1}\right)} \leq C,
\end{equation}
and hence
\begin{equation}\label{2.9}
\frac{I}{C} \leq D^{2} v_{R} \leq C I \quad \text { on } B_{1.1},
\end{equation}
for some $C$ independent of $R$. This tells us the result (\ref{HolderRegularity_MA}).
More explicitly, by the definition of $v_R$, we see that
$
D^2v_R(y)=D^2v(x+\frac{R}{4}y)
$ are bounded uniform to $x\in B_{R_1}^c$, hence
$$
||D^2v||_{C^0(B_{R_1}^c)}\leq \sup_{x\in B_{R_1}^c}
||D^2v||_{C^0(B_{\frac{|x|}{4}}(x))}\leq C.
$$
Since $f\in C^{\alpha}(B_{2R_1})$, the interior  $C^{2,\alpha}$ estimate by Caffarelli \cite{7}, Jian and Wang \cite{25} tells us
\begin{equation}\label{temp-22}
||D^2v||_{C^{\alpha}(B_{2R_1})}\leq C.
\end{equation}
This provides us the boundedness of Hessian matrix $D^2v$ and hence so is $D^2u$.
Similarly, for H\"older semi-norm, by definition we have
$$
\begin{array}{llll}
[D^2v]_{C^{\alpha}(\mathbb{R}^n)} &=&\displaystyle \sup_{x_1,x_2\in \mathbb{R}^n}\dfrac{|D^2v(x_1)-D^2v(x_2)|}{|x_1-x_2|^{\alpha}}\\
&\leq &\displaystyle
\max\left\{
\sup_{x_1,x_2\in B_{2R_1}
}\dfrac{|D^2v(x_1)-D^2v(x_2)|}{|x_1-x_2|^{\alpha}}\right.,\\
&&\displaystyle\left.
\sup_{
x_1\in B_{R_1}^c\atop
x_2\in B_{\frac{1}{4}|x_1|}(x_1)
}\dfrac{|D^2v(x_1)-D^2v(x_2)|}{|x_1-x_2|^{\alpha}}
,\sup_{
x_1\in B_{R_1}^c\atop
x_2\in B_{\frac{1}{4}|x_1|}(x_1)^c
}\dfrac{|D^2v(x_1)-D^2v(x_2)|}{|x_1-x_2|^{\alpha}}
\right\}.\\
\end{array}
$$
The first term bounded due to $C^{2,\alpha}$ estimate in (\ref{temp-22}). For the second term, note that
$$
\begin{array}{llllll}
\displaystyle\sup_{
x_1\in B_{R_1}^c\atop
x_2\in B_{\frac{1}{4}|x_1|}(x_1)
}\dfrac{|D^2v(x_1)-D^2v(x_2)|}{|x_1-x_2|^{\alpha}}
&\leq &\displaystyle\sup_{x\in B_{R_1}^c} [D^2v]_{C^{\alpha}(B_{\frac{1}{4}|x|}(x))}\\
&=& \displaystyle\sup_{x\in B_{R_1}^c}
\sup_{y_1,y_2\in B_{\frac{1}{4}|x|}(x)}
\dfrac{|D^2v(y_1)-D^2v(y_2)|}{|y_1-y_2|^{\alpha}}\\
&= & \displaystyle
\sup_{x\in B_{R_1}^c}
\sup_{z_1,z_2\in B_{1}}
\dfrac{|D^2v(x+\frac{|x|}{4}z_1)-D^2v(x+\frac{|x|}{4}z_2)|}{
(\frac{|x|}{4})^{\alpha}
\cdot|z_1-z_2|^{\alpha}}\\
&\leq &\displaystyle
\sup_{x\in B_{R_1}^c\atop R=|x|} [D^2v_R]_{C^{\alpha}(B_1)}\cdot (\frac{4}{R})^{\alpha}\\
&\leq & C,\\
\end{array}
$$
for some constant $C$ from (\ref{CalphaEstimate}). For the third term, due to Hessian matrix $D^2v$ is proved to be bounded, hence
$$
\sup_{
x_1\in B_{R_1}^c\atop
x_2\in B_{\frac{1}{4}|x_1|}(x_1)^c
}\dfrac{|D^2v(x_1)-D^2v(x_2)|}{|x_1-x_2|^{\alpha}}
\leq (\dfrac{4}{3R_1})^{\alpha}\cdot 2||D^2v||_{C^0(\mathbb{R}^n)}\leq C.
$$
Combine these three parts and since the linear transform $T$ from Lemma \ref{roughestimate} doesn't degenerate, we obtain (\ref{HolderRegularity_MA}) immediately.

Using Newton-Leibnitz formula between $\operatorname{det}D^{2} \overline{v}_{R}(y)=f_{1, R}(y)$ and $\operatorname{det}I=1$ gives
\begin{equation*}
\widetilde{a}_{i j} \partial_{i j} w_{R}=f_{1, R}(y)-1=O\left(R^{-\min\{\beta,\gamma\}}\right),
\end{equation*}
where $\widetilde{a}_{i j}(y)=\int_{0}^{1} \operatorname{cof}_{i j}\left(I+t D^{2} w_{R}(y)\right) d t$.

By (\ref{CalphaEstimate}), (\ref{2.9}) and using Landau-Kolmogorov interpolation inequality, since $v_R$ is bounded on $B_2$, there exists some constant $C$ independent of $R$ such that $\left\|v_{R}\right\|_{C^{2, \alpha}\left(\bar{B}_{1}\right)}\leq C$., the $C^{\alpha}$ norm of $D^2w_R$ is also bounded by some  constant $C$ independent of $R$, hence
\begin{equation*}
\frac{I}{C} \leq \widetilde{a}_{i j} \leq C I, \text { on } B_{1.1}, \quad\left\|\widetilde{a}_{i j}\right\|_{C^{ \alpha}\left(\overline{B}_{1.1}\right)} \leq C.
\end{equation*}
Thus Schauder's estimate (note that Theorem 6.2 of \cite{GT}  demands the coefficients have bounded H\"older norm, then the frozen coefficient method can be applied) gives us
\begin{equation*}
\left\|w_{R}\right\|_{C^{2, \alpha}\left(B_{1}\right)} \leq C\left(\left\|w_{R}\right\|_{L^{\infty}\left(\overline{B}_{1,1}\right)}+\left\|f_{1, R}-1\right\|_{C^{\alpha}\left(\overline{B}_{1}\right)}\right) \leq C R^{-\min\{\varepsilon,\beta,\gamma\}}.
\end{equation*}
Scale back to the original $B_{\frac{|x|}{4}}(x)$ ball, this is exactly our result (\ref{ConvergeRateofHessian}).
\end{proof}

\begin{remark}\label{corollary_estimate2}

The condition (\ref{Holder-Condition}) can be replaced by the following
 \begin{equation}\label{1-orderCondition}
  \exists\gamma>0\text{ such that }
  \limsup_{|x|\rightarrow\infty}|x|^{1+\gamma}
  \left| Df ( x ) \right| <\infty.
  \end{equation}

Condition (\ref{1-orderCondition}) demands $f$ has at least one order derivative while condition (\ref{Holder-Condition}) only demands the $C^{\alpha}$-semi norm of $f$ has a vanishing speed at infinity.

For example, similar to the example in Remark \ref{example}, we take $f(x):=e^{-|x|}\sin(e^{|x|})+1$. On the one hand,  $Df(x)$ doesn't admit a limit at infinity, hence $f$ doesn't satisfies condition (\ref{1-orderCondition}).

On the other hand, we calculate its $C^{\alpha}$ H\"older semi-norm directly
\begin{equation*}
\begin{array}{llllll}
[f]_{C^{\alpha}\left({B_{\frac{3}{2} R} \backslash B_{\frac{R}{2}}}\right)}&=&\displaystyle\sup _{z_{1}, z_{2} \in B_{\frac{3R}{2}}\setminus B_{\frac{R}{2}}} \frac{\left|f\left(z_{1}\right)-f\left(z_{2}\right)\right|}{\left|z_{1}-z_{2}\right|^{\alpha}}\\
& \leq &
\displaystyle\sup _{z_{1}, z_{2} \in B_{\frac{3R}{2}}\setminus B_{\frac{R}{2}}} \left( e^{-|z_2|}\frac{\left|\sin(e^{|z_1|})-\sin(e^{|z_2|})\right|}{\left|z_{1}-z_{2}\right|^{\alpha}}
+ \sin(e^{|z_1|}) \frac{\left|e^{-|z_1|}-e^{-|z_2|}\right|}{\left|z_{1}-z_{2}\right|^{\alpha}}\right)\\
&\leq C & \displaystyle\sup _{z_{1}, z_{2} \in B_{\frac{3R}{2}}\setminus B_{\frac{R}{2}}}
\left(
e^{-R}\cdot\dfrac{\max_{z\in B_{\frac{3R}{2}}\setminus B_{\frac{R}{2}}}|\cos(e^{|z|})|\cdot |z_1-z_2|}{|z_1-z_2|^{\alpha}}\right.\\
&+&\displaystyle\left. 1\cdot \dfrac{\sup_{z\in B_{\frac{3R}{2}}\setminus B_{\frac{R}{2}}}|e^{-|z|}|\cdot |z_1-z_2|}{|z_1-z_2|^{\alpha}}
\right)\\
&\leq & Ce^{-R}\cdot R^{1-\alpha},\\
\end{array}
\end{equation*}
for some constant $C$ independent of $R$.  Hence  $f$ satisfies condition (\ref{Holder-Condition}) for all $\alpha\in(0,1),\ \gamma>0$.
\end{remark}

The study of equation (\ref{equation_NONConstant}) is transformed into Monge-Amp\`ere type by taking Legendre transform as formula (\ref{LegendreTransform}).
We can easily verify as in Section \ref{Section-ConstantSituation} that  $\bolangxian{u}(\bolangxian{x})$ satisfies the following Monge-Amp\`ere type equation $$
\sum _ { i = 1 } ^ { n } \ln \bolangxian{\lambda_i}=\frac{2b}{\sqrt{a^2+1}}f(\dfrac{1}{2b}(\bolangxian{x}-D\bolangxian{u}(\bolangxian{x}))).
$$
This equation is equivalent (under the condition of $\bolangxian{\lambda_i}>0$) to
\begin{equation}\label{transformed}
\operatorname { det } D ^ { 2 } \widetilde { u } = \exp \left\{\frac{2b}{\sqrt{a^2+1}} f \left( \frac { 1 } { 2 b } ( \widetilde { x } - D \widetilde { u } ( \widetilde { x } ) ) \right) \right\}=:g(\bolangxian{x}).
\end{equation}

\begin{theorem}\label{LimitofHessian}
  Let $u \in C ^ {4 } \left( \mathbb { R } ^ { n } \right)$ be a classical solution of (\ref{equation_NONConstant})
and satisfy (\ref{boundedHessian}).
  Also, we assume that $f$ satisfies
   condition (\ref{0-orderCondition}), (\ref{1-orderCondition})
   for some $f(\infty)<0$ .
  Then there exist $A\in\mathcal{A}(f(\infty),-a+b), \alpha > 0 , C \left( n , f , \beta , \gamma \right),$ and $ R_2(n,f,\beta,\gamma)$ such that \begin{equation}\label{Result_LimitofHessian}
  \left| D ^ { 2 } u ( x ) - A \right| \leq \frac { C } { | x | ^ {\alpha} } , \quad \forall | x | \geq R _ {2}.
  \end{equation}
\end{theorem}

\begin{proof}[Proof.]
By condition (\ref{boundedHessian}), $-a+b<\lambda_i\leq M$ for $i=1,2,\cdots,n$, hence for some $\delta=\delta(M)>0$,
\begin{equation}\label{smalldelta}
\bolangxian{\lambda_i}=1-\dfrac{2b}{\lambda_i+a+b}\quad\text{satisfies}\quad 0<\bolangxian{\lambda_i}<1-\delta.
\end{equation}
Also note that $f\in C^0(\mathbb{R}^n)$ has a finite lower bound. Due to $\ln \bolangxian{\lambda_i}<0$ holds for all $i=1,2,\cdots,n$. Hence similar to the strategy used in \cite{Wang.Chong-paper}, we naturally have a lower bound such that for all $i=1,2,\cdots,n$,  $$
\ln\bolangxian{\lambda_i}>\frac{2 b}{\sqrt{a^{2}+1}}\inf_{\mathbb{R}^n}f>-\infty,\quad\text{i.e.}\quad
\bolangxian{\lambda_i}>\exp\{\frac{2 b}{\sqrt{a^{2}+1}}\inf_{\mathbb{R}^n}f\}>0.
$$
Combine these two results, we have $$
0<\delta<\bolangxian{\lambda_i}<1-\delta,
$$
for some $\delta=\delta(a,b,\inf_{\mathbb{R}^n}f,M)$.
Furthermore, we have a reversed direction of formula (\ref{limitofX}). By the definition of Legendre transform (\ref{LegendreTransform}), we have
\begin{equation*}
2b|x|=|\widetilde{x}-D \widetilde{u}(\widetilde{x})| \geq|\widetilde{x}|-|D \widetilde{u}(\widetilde{x})| \geq|\widetilde{x}|-(1-\delta)|\widetilde{x}|-|D \widetilde{u}(0)|,
\end{equation*}
i.e. \begin{equation}\label{limitofX2}
|x|\geq \frac{\delta}{2b}|\widetilde{x}|-\frac{1}{2b}|D\widetilde{u}(0)|.
\end{equation}
Combine formula (\ref{limitofX}) and (\ref{limitofX2}), we see that there exists some constant $C_0=C_0(\inf_{\mathbb{R}^n}f,a,b,M)$ such that
\begin{equation}\label{linear-of-X}
\frac{1}{C_0}|x|\leq |\widetilde{x}|\leq C_0|x|,\quad\text{for}\quad|x|\gg 1.
\end{equation}
This linear growth equivalence result (\ref{linear-of-X}) is the major reason that we demand the boundedness of Hessian.

We obtain the limit of Hessian for equation (\ref{transformed}) by Theorem \ref{corollary_estimate3} and Remark \ref{corollary_estimate2} first. Thus we need to verify the asymptotic behavior of $g(\bolangxian{x})$.

\textit{Step 1  Verify $g(\widetilde{x})$ satisfies condition (\ref{0-orderCondition})}

By the equivalence (\ref{linear-of-X}), $$
  \lim_{\bolangxian{x}\rightarrow\infty}g(\bolangxian{x})=\exp\{\frac{2 b}{\sqrt{a^{2}+1}}f(\infty)\}=:g(\infty)\in(0,1).
  $$
  Hence we have
  $$
  |\bolangxian{x}|^{\beta}|g(\bolangxian{x})-g(\infty)|
 = e^{\frac{2 b}{\sqrt{a^{2}+1}}}
  \dfrac{e^{f(\infty)}|\bolangxian{x}|^{\beta}}{\left|\frac{\bolangxian{x}-D\bolangxian{u}(\bolangxian{x})}{2b}\right|^{\beta}}
  \cdot\left|\frac{\bolangxian{x}-D\bolangxian{u}(\bolangxian{x})}{2b}\right|^{\beta}
  \cdot \left|
  e^{f(\frac{\bolangxian{x}-D\bolangxian{u}(\bolangxian{x})}{2b})-f(\infty)}-1
  \right|.
  $$
  Note that $f$ is a bounded function and admits a limit $f(\infty)$ at infinity, together with formula (\ref{linear-of-X}) tell us $f(\frac{\bolangxian{x}-D\bolangxian{u}(\bolangxian{x})}{2b})-f(\infty)$ is less than 1 for sufficiently large $|\bolangxian{x}|$.

  Due to $|e^t-1|\leq e|t|$ as long as $|t|\leq 1$, by $\widetilde{x}-D\widetilde{u}(\widetilde{x})=2bx$ and equivalence (\ref{linear-of-X}), we have$$
  \limsup_{|\bolangxian{x}|\rightarrow\infty}
  |\bolangxian{x}|^{\beta}|g(\bolangxian{x})-g(\infty)|\leq C_0^{\beta}
  e^{\frac{2 b}{\sqrt{a^{2}+1}}+f(\infty)+1}
  \limsup_{|x|\rightarrow\infty}
  |x|^{\beta}\left|f(x)-f(\infty)\right|
  <\infty.
  $$

\textit{Step 2 Verify $g(\widetilde{x})$ satisfies condition (\ref{1-orderCondition})}

  By taking derivative once, we have $$
  Dg(\bolangxian{x})=\exp\{\frac{2 b}{\sqrt{a^{2}+1}}f\}\cdot Df(\dfrac{1}{2b}(\bolangxian{x}-D\bolangxian{u}(\bolangxian{x})))
  \cdot\dfrac{1}{2b}[I-D^2\bolangxian{u}].
  $$

  Due to $D^2\widetilde{u}$  and $\exp\{\frac{2 b}{\sqrt{a^{2}+1}}f\}$ are bounded (for sufficiently large $\bolangxian{x}$), hence we only need to consider the rest part of $Dg(\widetilde{x})$.

  By equivalence (\ref{linear-of-X})  we obtain $$
|\bolangxian{x}|^{\gamma+1}
  \left|Df(\dfrac{1}{2b}(\bolangxian{x}-D\bolangxian{u}(\bolangxian{x})))\right|
  \leq C_0^{\gamma+1}
  |x|^{\gamma+1}\cdot |Df(x)|.
  $$
  Take the limit $|\widetilde{x}|\rightarrow\infty$ and condition (\ref{1-orderCondition}) gives us $$
  \limsup_{|\widetilde{x}|\rightarrow\infty}|\bolangxian{x}|^{\gamma+1}
  \left||Df(\dfrac{1}{2b}(\bolangxian{x}-D\bolangxian{u}(\bolangxian{x})))\right|
  \leq C_0^{\gamma+1}
  \limsup_{|x|\rightarrow\infty}|x|^{\gamma+1} \cdot|D f(x)|<\infty.
  $$

By Theorem \ref{corollary_estimate3} and Remark \ref{corollary_estimate2} , we have $$
\left| D ^ { 2 } \bolangxian{u} ( \bolangxian{x} ) - \bolangxian{A} \right| \leq \frac { C } { | \bolangxian{x} | ^ { \min \{ \beta , \varepsilon,\gamma\} } } , \quad \forall | \bolangxian{x} | \geq R _ { 1 }
$$
for some $\bolangxian{A} \in \bolangxian{\mathcal {A }} : = \left\{ \bolangxian{A} \in \operatorname { Sym } ( n ) :
 \det \bolangxian{A}=g(\infty),\ \bolangxian{A}>0
 \right\}$ and $\varepsilon > 0 , C ( n , f , \beta , \gamma ) , R _ { 1 } ( n , f , \beta , \gamma )$.

Now we prove that from the definition  of Legendre transform (\ref{LegendreTransform}) and the equivalence (\ref{linear-of-X}), there exists an $A\in\mathcal{A}(f(\infty),-a+b)$
such that (\ref{Result_LimitofHessian}) holds with
 $\alpha=\min\{\beta,\varepsilon,\gamma\}>0$.

By strip argument as in the proof of Theorem \ref{ConstantRHS1}, $D^2\bolangxian{u}, \bolangxian{A}$ are away from the origin and identity matrix (by a positive distance $\delta$).
Take $$
A:=\left( \frac { 1 } { 2 b } \left( I - \bolangxian{A} \right) \right) ^ { - 1 } - ( a + b ) I,
$$
which satisfies $F_{\tau}(A)=f(\infty),$
 we obtain $$
 \begin{array}{llll}
\left| D ^ { 2 } u - A \right|& =&2b\left|
\left( I - D ^ { 2 } \widetilde { u } ( \widetilde { x } )\right) ^ { - 1 }-
\left( I - \bolangxian{A} \right) ^ { - 1 }
\right|\\
&\leq &C_{\delta}|D^2\bolangxian{u}(\bolangxian{x})-\bolangxian{A}|\\
&
\leq& \dfrac{C}{|\bolangxian{x}|^{\min\{\beta,\varepsilon,\gamma\}}}.\\
\end{array}
$$
By the equivalence (\ref{linear-of-X}), for the same $\alpha>0$ as above,
we have formula (\ref{Result_LimitofHessian}).
\end{proof}

Moreover, we have not only  the limit of Hessian at infinity, but also the $C^{\alpha}$ bound of Hessian.
\begin{theorem}\label{HolderRegularity_ofHessian}
  Under the conditions as in Theorem \ref{LimitofHessian}, there exists $C=C(n,f,\beta,\gamma,\alpha,a,b,M)$ such that $$
  ||D^2u||_{C^{\alpha}(\mathbb{R}^n)}\leq C.
  $$
\end{theorem}
\begin{proof}[Proof.]
  Again, we apply Legendre transform (\ref{LegendreTransform}) as in Theorem \ref{LimitofHessian} to obtain (\ref{transformed}). By Theorem \ref{corollary_estimate3}, there exists some constant C relying on $\bolangxian{u}$ such that $$
  ||D_{\bolangxian{x}}^2\bolangxian{u}||_{C^{\alpha}(\mathbb{R}^n)}\leq C.
  $$

  Now we transform back this result to $D^2u$. From formula (\ref{LegendreTransform}) it follows that
  \begin{equation*}
D^{2} u(x)=\left(\frac{1}{2 b}\left(I-D_{\widetilde{x}}^{2} \widetilde{u}(\widetilde{x})\right)\right)^{-1}-(a+b) I
\end{equation*}
and hence for any $x,y\in\mathbb{R}^n$,
\begin{equation*}
\left|D^{2} u(x)-D^{2} u(y)\right|=2 b\left|\left(I-D_{\widetilde{x}}^{2} \widetilde{u}(\widetilde{x})\right)^{-1}-\left(I-D_{\widetilde{x}}^{2} \widetilde{u}(\widetilde{y})\right)^{-1}\right|.
\end{equation*}
As in the argument in Theorem \ref{LimitofHessian}, there exists some  $0<\delta$ such that
\begin{equation*}
0<\delta I \leq D_{\widetilde{x}}^{2} \widetilde{u}(\widetilde{x}) \leq(1-\delta) I.
\end{equation*}
Thus it follows that $\exists C_{1}=C_{1}(n, \delta)>0, \quad C_{2}=C_{2}(n, \delta)>0$ such that
\begin{equation}\label{equivalentHessian}
2 b \cdot C_{1}(\delta) \cdot\left|D_{\widetilde{x}}^{2} \widetilde{u}(\widetilde{x})-D_{\widetilde{x}}^{2} \widetilde{u}(\widetilde{y})\right| \geq\left|D^{2} u(x)-D^{2} u(y)\right| \geq 2 b \cdot C_{2}(\delta) \cdot\left|D_{\widetilde{x}}^{2} \widetilde{u}(\widetilde{x})-D_{\widetilde{x}}^{2} \widetilde{u}(\widetilde{y})\right|
\end{equation}

Combine formula (\ref{equivalentHessian}) and the equivalence (\ref{linear-of-X}), we see that $D^2u$ has bounded $C^{\alpha}$ semi-norm if and only if $D^2\bolangxian{u}$ has bounded $C^{\alpha}$ semi-norm.
\end{proof}

Due to the important normalization lemma of John-Cordoba and Gallegos (see \cite{11}) cannot be applied to classical special Lagrangian equation without changing the operator , the level set method developed in \cite{CL} cannot be applied easily to $\tau>\frac{\pi}{4}$ situation.

\section{Analysis of Linearized Equation}\label{AnalyzeLinearEquation}

Now that we have obtained the limit of Hessian at infinity and it converge with a H\"older decay speed, it is time to capture the linear and constant part of the solution. In order to do this, we follow the line of Li-Li-Yuan \cite{ExteriorLiouville} and analyze the linearized equation of (\ref{equation_NONConstant}).

By the characterise result of ellipticity and concavity structure of $F(\lambda(D^2u))$ type equation in \cite{CNS}, equation (\ref{equation_NONConstant})  is uniformly elliptic and concave.
For any direction $e\in\partial B_1$, we apply $\partial_e,\ \partial_{ee}$ to the equation $F_{\tau}(\lambda)=f(x)$, then we have,
  \begin{equation}\label{linearized-equation}
   D_ { M _ { i j } } F_{\tau}\left( D ^ { 2 }u \right) D _ { i j } \left( u _ { e } \right) = f _ { e } ( x ) , \quad \text { and } \quad D_ { M _ { i j } } F_{\tau}\left( D ^ { 2 } u \right) D _ { i j } \left( u _ { e e } \right) \geq f _ { e e } ( x ),
  \end{equation}
where $D_{M_{ij}}F_{\tau}(D^2u)$ stands for the value of partial derivative of $F_{\tau}(M)$ with respect to the $i$-line, $j$-column position of $M$ at $D^2u(x)$.

The major difference between the linearized equation we obtained here and the one in Li-Li-Yuan \cite{ExteriorLiouville} is that our equation is non-homogeneous. We take advantage of the linearity and use Green's function to transform the non-homogeneous equation into homogeneous equation. Then by studying the decay speed of solution of non-homogeneous equation and applying the theory in \cite{ExteriorLiouville} we obtain corresponding results as in Theorem \ref{exteriorLiouville} and Corollary \ref{ExteriorLiouville_NonPositive}.

We consider the following Dirichlet Problem
  \begin{equation}\label{Dirichlet}
  \left\{ \begin{array} { l l } {Lw:= a _ { i j } ( x ) D _ { i j } w ( x ) = f ( x ) ,} & { \text { in } \mathbb { R } ^ { n } ,} \\ { \lim _ { | x | \rightarrow \infty } w ( x ) = 0 ,} & { \text { as } | x | \rightarrow \infty ,} \end{array} \right.
  \end{equation}
 where the coefficients are $C^2$ (which is provided by $u\in C^4$), satisfy  \begin{equation}\label{HolderCoefficient}
  ||a_{ij}||_{C^{\alpha}(\mathbb{R}^n)}\leq M,
  \end{equation}
  for some $\alpha>0,M<\infty$,
strictly elliptic for some $\gamma>0$, \begin{equation}\label{EllipticCefficient}
a _ { i j } ( x ) \xi _ { i } \xi _ { j } \geq \gamma|\xi|^2, \quad \text { for all } \quad x , \xi \in \mathbb { R } ^ { n },
  \end{equation}
  and for some symmetric matrix $a_{ij}(\infty)$ , $\varepsilon_0>0,\ C<\infty$, \begin{equation}\label{short-RangeCoefficient}
  |a_{ij}(x)-a_{ij}(\infty)|\leq C|x|^{-\varepsilon_0}.
  \end{equation}

By using the  criterion in \cite{Equivalence} together with the Theorem 2.2 of \cite{ExteriorLiouville}, we see that under these conditions, the Green's function of operator $L$ is equivalent to the Green's function of Laplacian. More precisely, let $G_L(x,y)$ be the Green's function centered at $y$ ,   there exists constant C such that
\begin{equation*}
C^{-1}|x-y|^{2-n} \leq G_{L}(x, y) \leq C|x-y|^{2-n},\quad\forall x\not=y,
\end{equation*}
\begin{equation}\label{Equivalence-Green}
\left|\partial_{x_{i}} G_{L}(x, y)\right| \leq C|x-y|^{1-n}, \quad i=1, \cdots, n,\quad\forall x\not=y,
\end{equation}
\begin{equation*}
\left|\partial_{x_{i}} \partial_{x_{j}} G_{L}(x, y)\right| \leq C|x-y|^{-n}, \quad i, j=1, \cdots, n,\quad \forall x\not=y.
\end{equation*}

Now we study the existence result of Dirichlet problem (\ref{Dirichlet})  and study its asymptotic behavior at infinity.
For the weak solution $u$ of a linear elliptic equation $a_{ij}(x)D_{ij}u(x)=f(x)$ in $\mathbb{R}^n$ hereinafter, we mean that for any $\varphi\in C_0^{\infty}(\mathbb{R}^n)$, $u\in W^{1,1}_{loc}(\mathbb{R}^n)$ satisfies $$
\int_{\mathbb{R}^n}D_i(a_{ij}(x)\varphi(x))D_ju(x)+f(x)\varphi(x)\d x=0.
$$
For the distribution solution $u$ of a linear elliptic equation $a_{ij}(x)D_{ij}u(x)=f(x)$ in $\mathbb{R}^n$ hereinafter, we mean that for any $\varphi\in\mathcal{S}(\mathbb{R}^n)$,
$$
\<a_{ij}(x)D_{ij}u(x),\varphi\>=\<f,\varphi\>,
$$
where
$$
\<a_{ij}(x)D_{ij}u(x),\varphi\>:=\int_{\mathbb{R}^n}
u(x)D_{ij}(a_{ij}(x)\varphi(x))\d x,\text{ and }
\<f,\varphi\>:=\int_{\mathbb{R}^n}f(x)\varphi(x)\d x.
$$
We can easily see from $C_0^{\infty}(\mathbb{R}^n)\subset\mathcal{S}(\mathbb{R}^n)$ that if $u$ is a $W^{1,1}_{loc}$ distribution solution, then it is also a weak solution.

\begin{lemma}\label{existence}
  Assume in addition that $f\in C^0(\mathbb{R}^n)$ satisfies for some $k\geq 2, \ 0<\varepsilon\ll1$,
  \begin{equation}
  \label{decayoff_1}
\limsup_{|x|\rightarrow+\infty} |x|^{k+\varepsilon}|f(x)|<\infty.
  \end{equation}
  Then there exists a weak solution $u\in C^1(\mathbb{R}^n)$ to the Dirichlet problem (\ref{Dirichlet}) and satisfies $$
  |u(x)|\leq C|x|^{2-k-\varepsilon},
  $$
  for some constant $C$.
\end{lemma}
\begin{proof}[Proof.]
 By the definition of Green's function, the equivalence result  and potential theory (or Calder\'on-Zygmund inequality),  the following convolution $$
G_L*f(x):=\int_{\mathbb{R}^n}G_L(x,y)f(y)\d y
$$ belongs to $W^{2,p}(\mathbb{R}^n)$ for sufficiently large $p>\frac{n}{k}$ and is a distribution solution of $L[w](x)=f(x)$ (See for example \cite{Adams,Ziemer}). Hence by the embedding theory, it is also a weak solution.  Now we only need to verify that it vanishes with the desired speed at infinity.
Following the line of \cite{BLZ} (Lemma 2.2, formula (2.21)), we find out that as long as $f(y)$ satisfies (\ref{decayoff_1})
, we have $$
  |G_{L}*f(x)|\leq G_{L}*[C|x|^{-k-\varepsilon}]\rightarrow0,\quad\text{as}\quad |x|\rightarrow\infty.
  $$

  In fact, this is proved in a standard way by separating the integral domain into the following three part $$
\begin{array} { l } { E _ { 1 } : = \left\{ y \in \mathbb { R } ^ { n } , \quad | y | \leq | x | / 2  \right\} ,} \\ { E _ { 2 } : = \left\{ y \in \mathbb { R } ^ { n }  , \quad | y - x | \leq | x | / 2  \right\}, } \\ { E _ { 3 }: =  \mathbb { R } ^ { n }  \backslash \left( E _ { 1 } \cup E _ { 2 } \right) .} \end{array}
$$
Thus
$$
|G_L*f(x)| \leq C\int_{E_1\cup E_2\cup E_3}\dfrac{1}{|x-y|^{n-2}}\cdot \dfrac{1}{|y|^{k+\varepsilon}}\d y.
$$

By our choice of $E_i$, it   follows immediately that $$
\int_{E_1}\dfrac{1}{|x-y|^{n-2}\cdot|y|^{k+\varepsilon}}\d y
\leq \int_{B_{\frac{|x|}{2}}}\dfrac{1}{|y|^{k+\varepsilon}}\d y\cdot \dfrac{1}{|\frac{|x|}{2}|^{n-2}},
$$
$$
=C_{n,k,\varepsilon}\cdot |x|^{n-k-\varepsilon}\cdot\dfrac{1}{|x|^{n-2}}=
O(|x|^{2-k-\varepsilon})
,\quad\text{as}\quad |x|\rightarrow\infty.
$$

Similarly, note that in $E_2$ case, $|y-x|\leq \frac{|x|}{2}\leq |y|$  we have
\begin{equation*}
  \int_{E_2}\dfrac{1}{|x-y|^{n-2}\cdot|y|^{k+\varepsilon}}\d y
  \leq C_k
  \int_{|x-y|\leq \frac{|x|}{2}}\dfrac{1}{|x-y|^{n-2+\varepsilon}}\d y
  \cdot \dfrac{1}{|x|^{k}},
\end{equation*}
 $$
\leq
\int_0^{\frac{|x|}{2}}\dfrac{1}{r^{\varepsilon-1}}\d r\cdot \dfrac{1}{|x|^k}=O(|x|^{-k-\varepsilon})
,\quad\text{as}\quad |x|\rightarrow\infty.
$$
Now we separate $E_3$ into two parts
$$
E_3^+:=\{y\in E_3:|x-y|\geq |y|\},\ E_3^-:=E_3\setminus E_3^+.
$$
Then
$$
\int_{E_3^+}\dfrac{1}{|x-y|^{n-2}\cdot|y|^{k+\varepsilon}}\d y
\leq \int_{|y|\geq \frac{|x|}{2}}
\dfrac{1}{|y|^{n+k+\varepsilon}}\d y
=O( |x|^{-k-\varepsilon}),\quad\text{as}\quad |x|\rightarrow\infty,
$$
and
$$
\int_{E_3^-}\dfrac{1}{|x-y|^{n-2}\cdot|y|^{k+\varepsilon}}\d y
\leq \int_{|y-x|\geq \frac{|x|}{2}}
\dfrac{1}{|y-x|^{n+k+\varepsilon}}\d y
=O( |x|^{-k-\varepsilon}),\quad\text{as}\quad |x|\rightarrow\infty.
$$

\color{black}
 Hence we immediately obtain  $$
 |w(x)|=|G_L*f(x)|\leq C|x|^{-k-\varepsilon},\quad\text{as}\quad |x|\rightarrow\infty.
 $$
Thus $w(x):=G_L*f(x)$ solves the Dirichlet problem (\ref{Dirichlet}) with $O(|x|^{2-k-\varepsilon})$ order vanishing at infinity.
\end{proof}

Combine Lemma \ref{existence} with the exterior Liouville theorem for homogeneous equation proved by Li-Li-Yuan \cite{ExteriorLiouville}, we have the exterior asymptotic behavior theory for a class of non-homogeneous equations. For reading simplicity, we recall the Theorem 2.2 of \cite{ExteriorLiouville} as the following.
\begin{theorem}\label{exteriorLiouville-ORI}
  Let $v$ be a positive classical solution of $
  a_{ij}(x)D_{ij}v(x)=0\quad\text{in}\quad\mathbb{R}^n\setminus\overline{B_1}
  $
, then there exists a constant $v_{\infty}\geq 0$ such that
  \begin{equation*}
v(x)=v_{\infty}+o\left(|x|^{2-n+\delta}\right) \quad \text { as }|x| \rightarrow \infty, \text { for all } \delta>0.
\end{equation*}
  Moreover, if $a_{ij}(x)$ satisfies condition
  (\ref{short-RangeCoefficient}),
  then \begin{equation}\label{Result_ExteriorLiouville}
  v ( x ) = v _ { \infty } + O \left( |x|^{2-n} \right)\quad \text{as}\quad | x | \rightarrow \infty.
  \end{equation}
\end{theorem}
Then the equivalence of Green's function gives us the following corresponding results for non-homogeneous situation.

\begin{theorem}\label{exteriorLiouville}
  Let $v$ be a positive classical solution of $
  a_{ij}(x)D_{ij}v(x)=f(x)\text{ in }\mathbb{R}^n\setminus\overline{B_1}
  $, the coefficients satisfy (\ref{short-RangeCoefficient}) and $f\in C^{0}(\mathbb{R}^n)$ satisfy
  (\ref{decayoff_1}) with $k=n$.
  Then there exists a constant $v_{\infty}\geq 0$ such that \begin{equation}\label{Result_ExteriorLiouville}
  v ( x ) = v _ { \infty } + O \left( |x|^{2-n} \right)\quad \text{as}\quad | x | \rightarrow \infty.
  \end{equation}
\end{theorem}
\begin{proof}[Proof.]
  Taking auxiliary function $w(x)$ as the solution we constructed in Lemma \ref{existence} satisfying system (\ref{Dirichlet}).
Then under the decay speed (\ref{decayoff_1}) of $f$ at infinity with $k=n$ , we have that $w=O(|x|^{2-n})$ as $|x|\rightarrow\infty$.

By the linearity of operator, we learned that $\widetilde{v}:=v-w$ is a $C^1$ weak solution of $
a _ { i j } ( x ) D _ { i j } \widetilde{v} ( x ) = 0
$
in exterior domain. Since the coefficients are uniformly elliptic and has bounded $C^{\alpha}$ semi-norm, interior Schauder estimate tells us $\widetilde{v}$ is a classical solution.
Also, we learned that the auxiliary function $w$ is bounded, hence $v\geq 0$ implies that there exists a constant lower bound to $\widetilde{v}$.

Applying the exteirior Liouville theorem for homogeneous linear elliptic equation i.e. Theorem \ref{exteriorLiouville-ORI}, we have
 $$
 \widetilde{v}( x ) = \widetilde{v} _ { \infty } + O \left( |x|^{2-n}\right)\quad \text { as }\quad | x | \rightarrow \infty,
$$
for some constant $\widetilde{v}_{\infty}$.
Due to  $v=\widetilde{v}+w$ is non-negative, the result follows immediately.
\end{proof}

\begin{corollary}\label{ExteriorLiouville_NonPositive}
  Let $v$ be a classical solution of $
  a_{ij}(x)D_{ij}v(x)=f(x)\quad\text{in}\quad\mathbb{R}^n\setminus\overline{B_1}
  $
  Suppose that $$
  |Dv(x)|=O(|x|^{-1}),\quad\text{as}\quad |x|\rightarrow+\infty.
  $$
  Also, we demand that $f\in C^{0}(\mathbb{R}^n)$  satisfies (\ref{decayoff_1}) with $k=n$.
  Then there exists a constant $v_{\infty}\geq 0$ such that (\ref{Result_ExteriorLiouville}) holds.
\end{corollary}
\begin{proof}[Proof.]
In virtue of Theorem \ref{exteriorLiouville}, we only need to prove that $v$ is bounded from at least on one side. Again, by the equivalence of Green's function, we transform this problem into homogeneous situation and apply Corollary 2.1 of \cite{ExteriorLiouville}.Let $w$ be the auxiliary function as in Lemma \ref{existence} satisfying (\ref{Dirichlet}), then $$
  w(x)=\int_{\mathbb{R}^n}G_L(x,y)f(y)\d y,\quad\text{and}\quad
  \partial_iw(x)=
  \int_{\mathbb{R}^n}\partial_{x_i}G_L(x,y)f(y)\d y.
  $$
Applying similar analysis as in Lemma \ref{existence}, it follows that $$
Dw(x)=O(|x|^{1-n}),\quad\text{as}\quad |x|\rightarrow+\infty.
$$

By triangle inequality, we set  $u:=v-w$ and obtain $$
|Du(x)|\leq |Dv(x)|+|Dw(x)|=O(|x|^{-1}),\quad\text{as}\quad |x|\rightarrow+\infty.
$$

Now we prove that $u$ is bounded from one side, then the asymptotic behavior of $w$ tells us $v$ is also bounded from one side.

Argue by contradiction, if $u$ were unbounded on both sides, there
would exist a sequence $\left\{x_{k}\right\}_{k=1}^{\infty},$ such that $1<\left|x_{k}\right|<\left|x_{k+1}\right| \rightarrow+\infty$ and $v\left(x_{k}\right)=0$ for all $k \in \mathbb{Z}^{+} .$ Then, it follows from $|D v(x)| \leq C /|x|$ (for all $x \in \mathbb{R}^n\setminus\overline{B_1} )$ that, for any
$k \in \mathbb{Z}^{+}$ and any $x \in \partial B_{\left|x_{k}\right|},$ we have
\begin{equation*}
|v(x)|\leq \frac{C}{\left|x_{k}\right|} \cdot 2 \pi\left|x_{k}\right|=2 C \pi.
\end{equation*}
 By the maximum principle, we   conclude that
$|v(x)| \leq 2 C \pi$ on $\overline{B}_{\left|x_{k+1}\right|} \backslash B_{\left|x_{k}\right|}$ for all $k \in \mathbb{Z}^{+} .$ Therefore, $|v(x)| \leq 2 C \pi$ on $\mathbb{R}^n\setminus\overline{B_{\left|x_{1}\right|}},$
contradicts the unboundedness assumption.

Translating $v$ by a constant doesn't affect the equation, so we can apply Theorem \ref{exteriorLiouville} and obtain the asymptotic results.
\end{proof}

\section{Proof of Theorem \ref{result1}}\label{ProofofResult1}

In this section, we provide the proof of Theorem \ref{result1} following the line of Li-Li-Yuan \cite{ExteriorLiouville}, some barrier functions are necessary to enhance the convergence speed from (\ref{Result_LimitofHessian}).

\begin{lemma}[$|x|^{-\frac{1}{2}}$ as a Barrier]\label{barrier1}If a smooth function $u$ satisfies the differential inequality
\begin{equation}\label{linearPDI}
  a_{ij}(x)D_{ij}u\geq g(x),\quad\text{in}\quad \mathbb{R}^n\setminus\overline{B_1},
\end{equation}
  and  $u\rightarrow 0,\ \text{as}\ |x|\rightarrow\infty$, where the coefficients are uniformly elliptic and satisfy $
  a_{ij}(x)\rightarrow a_{ij}(\infty),\ \text{as}\  |x|\rightarrow\infty $,
  for some positive symmetric matrix $a_{ij}(\infty)$.
  Suppose that $$
  \limsup_{|x|\rightarrow\infty}|x|^{\frac{3}{2}}|g(x)|<\infty.
  $$
  Then for some constant $C$, $
  u(x)\leq C|x|^{-\frac{1}{2}},\ \forall |x|\geq 2.
  $
\end{lemma}
\begin{proof}[Proof.]

By suitable change of coordinate as in Lemma 6.1 of \cite{GT}, we may assume without loss of generality that $a_{ij}(\infty)=\delta_{ij}$. Since the change of coordinate only relies on $a_{ij}(\infty)$, which is bounded w.r.t $x$ variable, hence the result still holds true by allowing $C$ relies on $a_{ij}(\infty)$.

To be more explicit, let $P$ be a positive constant matrix such that $\widetilde{a_{ij}}(\infty)=P^Ta_{ij}(\infty)P$ is diagonal matrix whose diagonal elements are the eigenvalues $\lambda_1,\cdots,\lambda_n$ of $a_{ij}(\infty)$. Furthermore, by taking $D=\mathtt{diag}(\lambda_{i}^{-1 / 2} \delta_{i j})$, we have
\begin{equation}\label{change-of-coordinate}
Q:=PD,\ \widetilde{u}(x):=u(xQ),\ [\widetilde{a_{ij}}(x)]:=D^TP^T[a_{ij}(xQ)]PD\rightarrow \delta_{ij}\ \text{as}\ |x|\rightarrow\infty.
\end{equation}
Since $P,D,Q$ only relies on $a_{ij}(\infty)$, if the result holds for $a_{ij}(\infty)=\delta_{ij}$, which means $\widetilde{u}(x)\leq C|x|^{-\frac{1}{2}}$, then it follows that $$
u(x)=\widetilde{u}(Q^{-1}x)\leq C|Q^{-1}x|^{-\frac{1}{2}}\leq C|x|^{-\frac{1}{2}}.
$$

Now we prove the result for  $a_{ij}(\infty)=\delta_{ij}$.
By linearity and comparison principle, we only need to proof that $|x|^{-\frac{1}{2}}$ forms a supersolution of this uniformly elliptic equation.
By direct calculating, we have
$$
 a_{ij}(x)D_{ij}(|x|^{-\frac{1}{2}})=
 (-\frac{1}{2}) \sum _ { i = 1 } ^ { n } a _ { i i } ( x ) | x | ^ { - \frac{3}{2}} + ( - \frac{3}{2} ) \cdot ( -\frac{1}{2} ) \sum _ { i , j = 1 } ^ { n } a _ { i j } ( x ) x _ { i } x _ { j } | x | ^ { - \frac{5}{2}}.
$$
Due to $a_{ij}(x)\rightarrow\delta_{ij}$ as $|x|\rightarrow\infty$,
$$
LHS=-\dfrac{2n-3}{4}|x|^{-\frac{3}{2}}+o(|x|^{-\frac{3}{2}}).
$$

Hence as long as $$
\lim_{|x|\rightarrow\infty}|x|^{k}|g(x)|<\infty,\quad\text{for\ some}\quad k\geq\dfrac{3}{2},
$$
we can pick a sufficiently large constant $C>0,R>2$ such that
$$
a_{ij}(x)D_{ij}(|x|^{-\frac{1}{2}})\leq -\frac{|g(x)|}{C},\quad\forall |x|\geq R.
$$
Then we can pick a even larger $C$ such that
$$
C|x|^{-\frac{1}{2}}>u(x),\quad\forall\ x\in\partial B_R.
$$
Comparison principle tells us $u(x)\leq C|x|^{-\frac{1}{2}} $ for  $R\leq |x|$.
Due to $u$ is a smooth function, hence it maintains bounded inside $B_R\setminus B_2$. The result follows immediately by picking sufficiently large $C$.
\end{proof}

\begin{lemma}[$| x | ^ { 2 - n } - | x | ^ { 2 - n - \varepsilon }$ as a Barrier ]\label{barrier2}If a smooth function $u$ satisfies (\ref{linearPDI}) and $u\rightarrow 0,\ \text{as}\ |x|\rightarrow\infty$,
  with coefficients uniformly elliptic and satisfy $$
  |a_{ij}(x)-a_{ij}(\infty)|\leq C|x|^{-\alpha},\quad\text{as}\quad |x|\rightarrow\infty,
  $$
  for some positive symmetric matrix $a_{ij}(\infty)$.
  If $0<\varepsilon<\alpha$ and$$
  g(x)\geq 0,\quad\text{or}\quad \limsup_{|x|\rightarrow\infty}|x|^n|g(x)|<\infty.
  $$
  Then for some constant $C$, $
  u(x)\leq C(|x|^{2-n}-|x|^{2-n-\varepsilon}),\ \forall |x|\geq  2.
  $
\end{lemma}
\begin{proof}[Proof.]

As the argument in Lemma \ref{barrier1}, we may assume without loss of generality that $a_{ij}(\infty)=\delta_{ij}$. Otherwise we use change of coordinate as in (\ref{change-of-coordinate}), if the result holds for $a_{ij}(\infty)=\delta_{ij}$, then $\forall |x|>2$, there exists constant $C$ such that $$
\begin{array}{llll}
  u(x)=\widetilde{u}(Q^{-1}x)&\leq & C(|Q^{-1}x|^{2-n}-|Q^{-1}x|^{2-n-\varepsilon})\\
  &\leq &C(|Q^{-1}|^{2-n}|x|^{2-n}-|Q^{-1}|^{2-n-\varepsilon}|x|^{2-n-\varepsilon}).\\
\end{array}
$$
If $|Q^{-1}|^{2-n}\leq |Q^{-1}|^{2-n-\varepsilon}$, then we immediately obtain $$
u(x)\leq C|Q^{-1}|^{2-n-\varepsilon}(|x|^{2-n}-|x|^{2-n-\varepsilon}).
$$
If $|Q^{-1}|^{2-n}> |Q^{-1}|^{2-n-\varepsilon}$, then for sufficiently large $|x|\gg 1$ there exists some constant $c_0$ such that $$
c_0|Q^{-1}|^{2-n}(|x|^{2-n}-|x|^{2-n-\varepsilon})\geq (|Q^{-1}|^{2-n}-|Q^{-1}|^{2-n-\varepsilon})|x|^{2-n-\varepsilon},
$$
hence we have $$
\begin{array}{lllll}
  u(x)&\leq& C|Q^{-1}|^{2-n}(|x|^{2-n}-|x|^{2-n-\varepsilon})+
(|Q^{-1}|^{2-n}-|Q^{-1}|^{2-n-\varepsilon})|x|^{2-n-\varepsilon}\\
&\leq &(C+c_0)|Q^{-1}|^{2-n}(|x|^{2-n}-|x|^{2-n-\varepsilon})
\end{array}
,\ \forall|x|\gg 1.
$$

Hence we only need to prove the result for $a_{ij}(\infty)=\delta_{ij}$.
Again, we verify $|x|^{2-n}-|x|^{2-n-\varepsilon}$ is a barrier by direct calculation
\begin{equation*}
  a^{ij}(x)D_{ij}|x|^{2-n}
  =(2-n)\sum_{i=1}^na_{ii}(x)|x|^{-n}
  +(-n)\cdot (2-n)
  \sum_{i,j=1}^na_{ij}(x)x_ix_j|x|^{-n-2}.
\end{equation*}
By the condition of coefficients, we have
\begin{equation*}
  |a^{ii}(x)-1|\leq \dfrac{C}{|x|^{\alpha}},\ \ |a^{ij}(x)-0|\leq \dfrac{C}{|x|^{\alpha}},\ \forall i\not=j,
\end{equation*}
hence
$$
\begin{array}{lllllll}
   a^{ij}(x)D_{ij}|x|^{2-n} & \leq & n(2-n)|x|^{-n}
  -n(2-n)|x|^{-n}
  +|(2-n)|\sum_{i=1}^n\dfrac{C}{|x|^{\alpha}}|x|^{-n}\\
  &+&|(-n)|\cdot|(2-n)|\sum_{i,j=1}^n\dfrac{C}{|x|^{\alpha}}x_ix_j|x|^{-n-2}\\
  &\leq &0+n\cdot|(2-n)|C
  |x|^{-n-\alpha}.\\
\end{array}
$$
Similarly we calculate
\begin{equation*}
  a^{ij}D_{ij}|x|^{2-n-\varepsilon}
  =(2-n-\varepsilon)
  \sum_{i=1}^na_{ii}(x)|x|^{-n-\varepsilon}
  +(-n-\varepsilon)\cdot (2-n-\varepsilon)
  \sum_{i,j=1}^na_{ij}(x)x_ix_j|x|^{-n-2-\varepsilon},
\end{equation*}
and hence
\begin{equation*}
  a^{ij}D_{ij}|x|^{2-n-\varepsilon}
  \geq -\varepsilon (2-n-\varepsilon)|x|^{-n-\varepsilon}
  -C_n|x|^{-n-\alpha-\varepsilon}.
\end{equation*}
Combining these two result we have
\begin{equation*}
  a^{ij}(x)D_{ij}
  (|x|^{2-n}-|x|^{2-n-\varepsilon})
  \leq C\cdot n|(2-n)|\cdot |x|^{-n-\alpha}
  +\varepsilon(2-n-\varepsilon)\cdot |x|^{-n-\varepsilon}+C_n|x|^{-n-\alpha-\varepsilon}.
\end{equation*}
As long as $\varepsilon<\alpha$, the negative term $\varepsilon(2-n-\varepsilon)|x|^{-n-\varepsilon}$ takes the lead and makes this less than $|g(x)|$ when $|x|$ is sufficiently large due to the constants here are universal.
Then comparison principle tells us the result.
\end{proof}

\begin{proof}[Proof of Theorem \ref{result1}]

By Theorem \ref{LimitofHessian}, there exist $A\in \mathcal{A}(f(\infty),-a+b),\ \alpha>0,\ R_2\gg1$ sufficiently large such that (\ref{Result_LimitofHessian}) holds.

Set  $
v : = u ( x ) - \frac { 1 } { 2 } x ^ { T } A x
$
then we have
$$
\left\{
\begin{array}{cc}
  F_{\tau}(D^2v+A)=F_{\tau}(D^2u)=f(x),\\
  F_{\tau}(A)=\lim_{|x|\rightarrow\infty}F_{\tau}(D^2u(x))=f(\infty).\\
\end{array}
\right.$$Hence $$
\bolangxian{f}(x):=f(x)-f(\infty)=\int_0^1D_{M_{ij}}F_{\tau}(tD^2v+A)\d t\cdot D_{ij}v=:\overline{a^{ij}}(x)D_{ij}v.
$$

Also, from taking derivatives to any $e\in\partial B_1$ direction, we consider the linearized equation as (\ref{linearized-equation}), $v_e$ and $v_{ee}$ satisfy
\begin{equation}\label{linearized-equation-2}
D_{M_{i j}} F_{\tau}\left(D^{2} v+A\right) D_{i j}\left(v_{e}\right)=f_{e}(x), \quad \text { and } \quad D_{M_{i j}} F_{\tau}\left(D^{2} v+A\right) D_{i j}\left(v_{e e}\right) \geq f_{e e}(x).
\end{equation}
Let $\hat{a^{ij}}(x)$ denote the coefficients $D_{M_{ij}}F_{\tau}(D^2v+A)$.

Due to $F_{\tau}$ is uniformly $C^2$ in a bounded subdomain of $\text{Sym}(n)$, hence for the same $\alpha$ from Theorem \ref{LimitofHessian}, there exists some constant $C$ such that $$
\left| \overline { a _ { i j } } ( x ) - D_{M_{ij}}F_{\tau} ( A ) \right| , \left| \widehat { a _ { i j } } ( x ) - D_{M_{ij}}F_{\tau} ( A) \right| \leq \frac { C } { | x | ^ { \alpha} }.
$$

Thus Theorem \ref{exteriorLiouville} and Corollary \ref{ExteriorLiouville_NonPositive} hold for the linear equations with coefficients $\overline{a_{ij}}(x),\hat{a_{ij}}(x)$.
From  condition (\ref{Low-Regular-Condition}) and Lemma \ref{barrier2},
we have the following finer convergence speed estimate $
v _ { e e } ( x ) \leq  C | x | ^ { 2 - n }
$.
hence
\begin{equation*}
\lambda_{\max }\left(D^{2} v\right)(x) \leq C|x|^{2-n},
\end{equation*}
and the ellipticity of equation (\ref{linearized-equation-2}) tells us
\begin{equation*}
\lambda_{\min }\left(D^{2} v\right)(x) \geq-C \lambda_{\max }\left(D^{2} v\right)(x)-C|\widetilde{f}(x)| \geq-C|x|^{2-n}.
\end{equation*}
Hence $$
\left| D ^ { 2 } v ( x ) \right| \leq C | x | ^ { 2 - n }.$$
Therefore
$$
\left| \overline { a  ^ { i j }} ( x ) - D_{M_{ij}}F_{\tau}(A) \right|\leq C | x | ^ { 2 - n } , \quad \forall x \in B _ { 1 } ^ { c }, $$
and
$$\left| \widehat { a  ^ { i j }} ( x ) - D_{M_{ij}}F_{\tau}(A) \right| \leq C | x | ^ { 2 - n } , \quad \forall x \in B _ { 1 } ^ { c }.
$$
Thus for any $e\in\partial B_1$ we have $$
\left| D v _ { e } \right| \leq \left| D ^ { 2 } v ( x ) \right| \leq C | x | ^ { 2 - n }  ,$$
which provides us ``bounded from one side'' condition. Since $F_{\tau}$ is uniformly $C^2$ on the range of $
\lambda(D^2u)$, hence Theorem \ref{HolderRegularity_ofHessian} tells us the coefficients $\overline{a^{ij}},\ \hat{a^{ij}}$ has bounded $C^{\alpha}$ norm.
Also note that the coefficients have a H\"older convergence speed and  $|Df|$ has $O(|x|^{-(n+\varepsilon)})$ order decay. These conditions enable us to apply Corollary \ref{ExteriorLiouville_NonPositive} to equation $\hat{a^{ij}}(x)D_{ij}(v_e)=f_e(x)$  and obtain
\begin{equation}\label{capture1-order}
\exists \beta _ { e } \in \mathbb { R } \text { s.t. } v _ { e } ( x ) = \beta_ { e } + O \left( | x | ^ { 2 - n } \right) , \quad a s | x | \rightarrow \infty.
\end{equation}

Picking $e_i:=(0,\cdots,0,\underbrace{1}_{\text{the $i$-th variable}},0,\cdots,0)$ and we use $\beta_i$ denote $\beta_{e_i}$ from formula (\ref{capture1-order}).

Set $\beta : = \left( \beta _ { 1 } , \beta _ { 2 } , \cdots , \beta _ { n } \right)\in\mathbb{R}^n$ and
$$\overline { v } ( x ): = v ( x ) - \beta ^ { T } x = u ( x ) - \left( \frac { 1 } { 2 } x ^ { T } A x + \beta ^ { T } x \right).$$
Then formula (\ref{capture1-order}) tells us $$
|D\overline{v}(x)|=|(\partial_1v-\beta_1,\cdots,\partial_nv-\beta_n)|=
O(|x|^{2-n}).
$$

Also note that it satisfies the linearized equation
\begin{equation*}
  \overline{a}^{ij}(x)D_{ij}\overline{v}=\overline{a}^{ij}(x)
  D_{ij}v=\widetilde{f}(x).
\end{equation*}
By the arguments above again and adapt Corollary \ref{ExteriorLiouville_NonPositive} we have
\begin{equation*}
  \exists \gamma\ s.t.\ \overline{v}(x)=\gamma+O(|x|^{2-n}),\ as\ |x|\rightarrow\infty.
\end{equation*}

Set
$
  Q(x)=\dfrac{1}{2}x^TAx+\beta^Tx+\gamma
$ and the formula above tells us
$$
  |u-Q|=|\overline{v}-\gamma|=O(|x|^{2-n}),\ \ as\ |x|\rightarrow\infty.
$$
\end{proof}
\section{Proof of Theorem \ref{result2}}\label{ProofOfResult2}

In this section, we provide an asymptotic behavior result that is parallel to the one by Bao-Li-Zhang \cite{BLZ}. From Theorem \ref{LimitofHessian}, we have proved that under suitable assumption of $f$, the Hessian matrix $D^2u$  admits a limit $A$ at infinity with $O(|x|^{-\alpha})$ order.

Instead of using the original results in \cite{BLZ}, we transform back using the definition of Legendre transform
and follow the line of proving Lemma 2.1 in \cite{BLZ} to obtain estimates of up to $(m+1)$-order derivative of solutions and H\"older estimate of $D^{m+1}u(x)$ directly.

\begin{lemma}\label{roughestimate_2}
  Under the conditions of Theorem \ref{result2}, suppose that there exists $\varepsilon>0$ such that $$
  |D^2u(x)-A|\leq c_1|x|^{-\varepsilon},\quad |x|\geq R_0.
  $$
  Assume without loss of generality that $u(0)=0, Du(0)=0$. Let $$
  w(x)=u(x)-\frac{1}{2}x^TAx,
  $$
  then there exist $C\left(n, R_{0}, \varepsilon, f(\infty), c_{1}, \zeta\right)>0$ and $R_{1}\left(n, R_{0}, \varepsilon, f(\infty), c_{1}, \zeta\right)>R_{0}$ such that for any
$\alpha \in(0,1)$,
\begin{equation*}
\left|D^{k} w(y)\right| \leq C|y|^{2-k-\varepsilon_{\zeta}}, \quad k=0, \ldots, m+1, \quad|y|>R_{1},
\end{equation*}
\begin{equation*}
\frac{\left|D^{m+1} w\left(y_{1}\right)-D^{m+1} w\left(y_{2}\right)\right|}{\left|y_{1}-y_{2}\right|^{\alpha}} \leq C\left|y_{1}\right|^{1-m-\varepsilon_{\zeta}-\alpha}, \quad\left|y_{1}\right|>R_{1},\ y_{2} \in B_{\frac{\left|y_{1}\right|}{2}}\left(y_{1}\right),
\end{equation*}
  where $\varepsilon _ { \zeta } := \min \{ \varepsilon , \zeta \}$.
\end{lemma}

\begin{proof}[Proof.]

  By direct calculate and our assumption $u(0)=0,\ Du(0)=\overrightarrow{0}$, it follows that for some $C(n,u|_{B_{R_0}},c_1)$
  \begin{equation}\label{2.6}
    |w(x)|\leq C|x|^{2-\varepsilon},\ \forall |x|\geq R_0.
  \end{equation}

  For $|x|=R>2 R_{0},$ let $$
u_{R}(y)=\left(\frac{4}{R}\right)^{2} u\left(x+\frac{R}{4} y\right),\quad\text{and}\quad
w_{R}(y)=\left(\frac{4}{R}\right)^{2} w\left(x+\frac{R}{4} y\right), \quad|y| \leq 2.
 $$
 Then the boundedness of $D^2u$ together with (\ref{2.6}) tell us$$
 \left\|u_{R}\right\|_{L^{\infty}\left(B_{2}\right)} \leq C, \quad\left\|w_{R}\right\|_{L^{\infty}\left(B_{2}\right)} \leq C R^{-\varepsilon},
 $$
 for some constant $C$ uniform to $R>2R_0.$

Now we attack the original equation by scaling back to unit size. It is easy to verify that $u_R$ satisfies \begin{equation}\label{temp-12}
F_{\tau}(\lambda(D^2u_R(y)))=f_{R}(y):=f(x+\frac{R}{4}y),\quad\text{in } B_2.
\end{equation}
Take difference with $F_{\tau}(\lambda(A))=f(\infty)$ gives us $$
\bolangxian{a_{ij}}(y)D_{ij}w_R=f_{R}(y)-f(\infty)=O(R^{-\zeta}),
$$
where $\bolangxian{a_{ij}}(y):=\int_0^1D_{M_{ij}}F_{\tau}(A+tD^2w_R(y))\d t$ is uniformly elliptic and Theorem \ref{HolderRegularity_ofHessian} tells us $\widetilde{a_{ij}}(y)$ have bounded $C^{\alpha}$ norm.

Apply the classical Schauder's estimate to obtain the $C^{2,\alpha}$ regularity of $w_R$
\begin{equation*}
\left\|w_{R}\right\|_{C^{2, \alpha}\left(\overline{B_{1}}\right)} \leq C\left(\left\|w_{R}\right\|_{L^{\infty}\left(\overline{B_{1.1}}\right)}+\left\|f_{ R}-f(\infty)\right\|_{C^{\alpha}\left(\overline{B}_{1}\right)}\right) \leq C R^{-\varepsilon_{\zeta}}.
\end{equation*}
Now we take derivative with respect to $e\in \mathbb{S}^{n-1}$ direction to equation (\ref{temp-12}) and obtain
\begin{equation}\label{temp-13}
D_{M_{ij}}F_{\tau}(D^2u_R(y))D_{ij}(\partial_eu_R(y))=\partial_ef_R(y),\quad\text{in }B_2.
\end{equation}
Since $D_{M_{ij}}F_{\tau}(D^2u_R(y)), \partial_{e} u_{R}$ and $\partial_{e} f_{R}$ are bounded in $C^{\alpha}$ norm, we have $$
||u_R||_{C^{3,\alpha}(\overline{B_1})}\leq C,
$$
which implies $$
||D_{M_{ij}}F_{\tau}(D^2u_R(y))||_{C^{3,\alpha}(\overline{B_1})}\leq C.
$$
By the definition of $w_R$, we see that $$
D^2w_R(y)=D^2w(x+\frac{R}{4}y)=D^2u(x+\frac{R}{4}y)-A=D^2u_R(y)-A,
$$
and hence $D_{ij}(\partial_eu_R(y))=D_{ij}(\partial_ew_R(y))$.
Hence equation (\ref{temp-13}) can also be written into
\begin{equation*}
D_{M_{i j}} F_{\tau}\left(D^{2} u_{R}(y)\right) D_{i j}\left(\partial_{e} w_{R}(y)\right)=\partial_{e} f_{R}(y), \quad \text { in } B_{2}.
\end{equation*}
We obtain by Schauder's estimate
\begin{equation*}
\left\|w_{R}\right\|_{C^{3, \alpha}\left(\overline{B_{1 / 2}}\right)} \leq C\left(\left\|w_{R}\right\|_{L^{\infty}\left(\overline{B_{3 / 4}}\right)}+| D f_{R} \|_{C^{\alpha}\left(\overline{{B}_{3 / 4}}\right)}\right) \leq C R^{-\varepsilon_{\zeta}},
\end{equation*}
which provides us $$
||D^3u_R||_{C^{\alpha}(\overline{B_{1/2}})}\leq CR^{-\varepsilon_{\zeta}}.
$$
Estimates on higher order derivatives can be obtained by further differentiation of the equation
and Schauder estimate. The result follows by induction immediately.
\end{proof}

Next we prove a bootstrap lemma to improve the estimates in Lemma \ref{roughestimate_2}. Lemma \ref{bootstrap} is originally proved for Monge-Amp\`ere equation in \cite{BLZ} and we work with general uniformly elliptic equation with bounded derivative of $F$ operator. Consider the uniformly elliptic equation \begin{equation}\label{temp-5}
  F(D^2u)=f(x),\quad\text{in}\quad\mathbb{R}^n\setminus B_{R_0},
  \end{equation}
  where $F$ is smooth up to boundary of $D:=\overline{D^2u(\mathbb{R}^n\setminus \overline{B_{R_0}})}$ i.e. the range of $D^2u$ and $f(x)\in C^m(\mathbb{R}^n\setminus\overline{B_{R_0}})$ satisfies
   condition (\ref{High-Regular-Condition}) for some $\zeta>2,\ m\geq 3$.
We  write $D_{M_{ij}}F(M)$ as $F_{M_{ij}}(M)$ for simplicity.
\begin{lemma}\label{bootstrap}Let $u$ be a solution of   equation  (\ref{temp-5}) described above.
  Let $v$ be a 2-order polynomial such that $F(D^2v)=f(\infty)$ and let $w:=u-v$.
  Suppose that for some $0<\varepsilon<\dfrac{1}{2}, \alpha\in(0,1)$, we have
  \begin{equation*}
\left|D^{k} w(x)\right| \leq C|x|^{2-\varepsilon-k}, \quad|x|>2 R_{1}, \quad k=0, \ldots, m+1,
\end{equation*}
\begin{equation*}
\frac{\left|D^{m+1} w\left(y_{1}\right)-D^{m+1} w\left(y_{2}\right)\right|}{\left|y_{1}-y_{2}\right|^{\alpha}} \leq C\left|y_{1}\right|^{1-m-\varepsilon-\alpha}, \quad\left|y_{1}\right|>2 R_{1}, y_{2} \in B_{\left|y_{1}\right| / 2}\left(y_{1}\right)
\end{equation*}
  Then
  \begin{equation*}
\left|D^{k} w(x)\right| \leq C|x|^{2-2 \varepsilon-k}, \quad|x|>2 R_{1}, \quad k=0, \ldots, m+1,
\end{equation*}
\begin{equation*}
\frac{\left|D^{m+1} w\left(y_{1}\right)-D^{m+1} w\left(y_{2}\right)\right|}{\left|y_{1}-y_{2}\right|^{\alpha}} \leq C\left|y_{1}\right|^{1-m-2 \varepsilon-\alpha}, \quad\left|y_{1}\right|>2 R_{1}, y_{2} \in B_{\left|y_{1}\right| / 2}\left(y_{1}\right).
\end{equation*}

\end{lemma}

\begin{proof}[Proof.]

Applying $\partial_k$ to   equation (\ref{temp-5}) and we obtain \begin{equation}\label{temp-linearized}
\hat{a_{ij}}D_{ij}(\partial_ku(x))=\partial_kf(x),
\end{equation}
where $\hat{a_{ij}}(x):=F_{M_{ij}}(D^2u(x)).
$

Then  this linearized equation is also uniformly elliptic with the coefficients satisfies (from the assumptions)$$
|\hat{a_{ij}}(x)-F_{M_{ij}}(A)|\leq ||D F||_{C^0(D)}\cdot \dfrac{C}{|x|^{\varepsilon}},\quad\text{and}\quad
\left| D \hat{a _ { i j }} ( x ) \right| \leq \frac { C } { | x | ^ { 1 + \varepsilon } } , \quad | x | > R _ { 1 }.
$$
Also, for the $\alpha$ in condition,  we have \footnote[2]{For simplicity, we use $a_{ij}$ stands for this $\hat{a_{ij}}$. Other linearize processes are similar to this.} $$
\frac { \left| D a _ { i j } \left( x _ { 1 } \right) - D a _ { i j } \left( x _ { 2 } \right) \right| } { \left| x _ { 1 } - x _ { 2 } \right| ^ { \alpha } } \leq C||D^2F||_{C^0(D)}\cdot  \left| x _ { 1 } \right| ^ { - 1 - \varepsilon - \alpha } , \quad \left| x _ { 1 } \right| > 2 R _ { 1 } , \quad x _ { 2 } \in B _ { \left| x _ { 1 } \right| / 2 } \left( x _ { 1 } \right).
$$

Apply $\partial_l$ to the (\ref{temp-linearized}) and set $h _ { 1 } = \partial _ { k l } u$ we further obtain$$
F_{M_{ij},M_{qr}}(D^2u)D_{ijk}uD_{qrl}u+
F_{M_{ij}}(D^2u)D_{ij}h_1=\partial_{kl}f(x).
$$
Due to the Hessian matrix converge to $D^2v$ (positive constant matrix) at infinity with H\"older speed $F_{M_{ij}}(D^2u)\rightarrow F_{M_{ij}}(D^2v)\ as\ |x|\rightarrow\infty$. The extremal matrix makes the operator also uniformly elliptic, hence the Green's function of $F_{M_{ij}}(D^2v)D_{ij}u$ exists and equivalence to  the one of Laplacian. Then we can write $$
F_{M_{ij}}(D^2v)D_{ij} h_1=f_2:=\partial_{kl}f-F_{M_{ij},M_{qr}}(D^2u)D_{ijk}uD_{qrl}u
-(F_{M_{ij}}(D^2u)-F_{M_{ij}}(D^2v))D_{ij}h_1.
$$

From the assumptions we stated, it follows that for the $\alpha\in (0,1)$ in the assumption,
\begin{equation*}
\left|f_{2}(x)\right| \leq C|x|^{-2-2 \varepsilon},\ \forall|x| \geq 2 R_{1},
\end{equation*}
\begin{equation*}
\frac{\left|f_{2}\left(x_{1}\right)-f_{2}\left(x_{2}\right)\right|}{\left|x_{1}-x_{2}\right|^{\alpha}} \leq \frac{C}{\left|x_{1}\right|^{2+2 \varepsilon+\alpha}}, \quad x_{2} \in B_{\left|x_{1}\right| / 2}\left(x_{1}\right),\left|x_{1}\right| \geq 2 R_{1}.
\end{equation*}

Note that this step demanded 3-order derivative of $f$. This is essential for this method due to it provides the H\"older continuity of $f_2$, which is used in Schauder estimate.

Note that $h_1\rightarrow D_{kl}v$ for fixed $k,l\in\{1,2,\cdots,n\}$, $h_1$ satisfies the following ``part'' Dirichlet problem
$$
\left\{
\begin{array}{cc}
  F_{M_{ij}}(D^2v) D_{ij}h_1=f_2, & in\ \mathbb{R}^n\setminus B_{R_1},\\
  h_1\rightarrow D_{kl}v, & as\ |x|\rightarrow\infty.\\
\end{array}
\right.
$$

The main target of this bootstrap lemma is to obtain a finer estimate on the convergence speed of Hessian matrix. From $O(|x|^{-\varepsilon})$ enhanced into $O(|x|^{-2\varepsilon})$ as long as $2\varepsilon<1$. In order to do this, we use a Green's function to change the equation into homogeneous situation.

Let $$h_2(x):=\int_{\mathbb{R}^n\setminus B_{R_1}}G_{F_{M_{ij}}(D^2v)}(x,y)f_2(y)\d y,$$ where $G_{F_{M_{ij}}(D^2v)}(x,y)$ is the distribution solution (Green's function) of $$
\left\{
\begin{array}{ccc}
  F_{M_{ij}}(D^2v)D_{ij}u(y)=\delta_x & \text{in }\mathbb{R}^n\\
  u(y)\rightarrow 0 & \text{as }|y|\rightarrow\infty\\
\end{array}
\right..
$$
The existence of such a Green's function and equivalence to fundamental solution is well-known, see \cite{Equivalence,Equivalence2} for example.
Then  $F_{M_{ij}}(D^2v) D_{ij}h_2=f_2$ and satisfies $h_2\rightarrow0\ as\ |x|\rightarrow\infty$.

Moreover, argue as in the proof of  Lemma \ref{existence}, separate $\mathbb{R}^n$ into three pieces $E_1, E_2, E_3$ and we obtain
$$
|D^jh_2(x)|\leq C | x | ^ { - 2 \varepsilon - j } , \quad | x | > 2 R _ { 1 } ,\quad\text{for } j = 0,1.
$$
Argue as in the proof of Lemma \ref{roughestimate_2}, we also obtain
\begin{equation*}
\frac{\left|D^{2} h_{2}\left(x_{1}\right)-D^{2} h_{2}\left(x_{2}\right)\right|}{\left|x_{1}-x_{2}\right|^{\alpha}} \leq \frac{C}{\left|x_{1}\right|^{2+2 \varepsilon+\alpha}}, \quad x_{2} \in B_{\frac{\left|x_{1}\right|}{2}}\left(x_{1}\right),\left|x_{1}\right|>2 R_{1}.
\end{equation*}
Indeed, for each $x _ { 0 } \in \mathbb { R } ^ { n } \backslash B _ { 2 R _ { 1 } } ,$ let $R = \left| x _ { 0 } \right| ,$ we set $$
h _ { 2 , R } ( y ) = h _ { 2 } \left( x _ { 0 } + \frac { R } { 4 } y \right) , \quad f _ { 2 , R } ( y ) = \frac { R ^ { 2 } } { 16 } f _ { 2 } \left( x _ { 0 } + \frac { R } { 4 } y \right) , \quad | y | \leq 2.
$$
Then we have$$
\left\| h _ { 2 , R } \right\| _ { C^0 \left(\overline{ B _ { 1 }} \right) } =||h_2||_{C^0
(\overline{B_{\frac{|x_0|}{4}}(x_0)})
}\leq C R ^ { - 2 \varepsilon },\quad\text{and}\quad
\left\| f _ { 2 , R } \right\| _ { C ^ { \alpha } \left( \overline{B _ { 1 }} \right) } \leq C R ^ { - 2 \varepsilon }.
$$
Hence Schauder Estimate tells us$$
\left\| h _ { 2 , R } \right\| _ { C ^ { 2 , \alpha } \left(
\overline{B _ { 1 }} \right) } \leq C \left( \left\| h _ { 2 , R } \right\| _ { L ^ { \infty } \left( B _ { 2 } \right) } + \left\| f _ { 2 , R } \right\| _ { C ^ { \alpha } \left(\overline{ B _ { 2 }} \right) } \right) \leq C R ^ { - 2 \varepsilon }.
$$
Meaning that we have obtained the desired result for $h_2$ i.e.
\begin{equation*}
\left|D^{j} h_{2}(x)\right| \leq C|x|^{-2 \varepsilon-j}, \quad \quad j=0,1,2, \quad|x|>2 R_{1},
\end{equation*}
and \begin{equation*}
\frac{\left|D^{2} h_{2}\left(x_{1}\right)-D^{2} h_{2}\left(x_{2}\right)\right|}{\left|x_{1}-x_{2}\right|^{\alpha}} \leq \frac{C}{\left|x_{1}\right|^{2}+2 \varepsilon+\alpha}, \quad x_{2} \in B_{\frac{|x_{1} |}{2}}\left(x_{1}\right),\left|x_{1}\right|>2 R_{1}.
\end{equation*}

Now we only need to study the difference between $h_1-D_{kl}v$ and $h_2$. By the linearity of Laplacian operator we have$$
\left\{
\begin{array}{cc}
  F_{M_{ij}}(D^2v) D_{ij}(h_1-D_{kl}v-h_2)=0, & in\ \mathbb{R}^n\setminus B_{2R_1},\\
  h_1-D_{kl}v-h_2\rightarrow 0, & as\ |x|\rightarrow\infty.\\
\end{array}
\right.
$$
By taking $|x|^{2-n}$ as a barrier function, it immediately follows that for some constant C$$
\left| h _ { 1 } ( x ) - D_ { k l }v - h _ { 2 } ( x ) \right| \leq C | x | ^ { 2 - n } , \quad | x | > 2 R _ { 1 }.
$$
Due to  $|x|^{2-n}$ converge faster than $|x|^{-2\varepsilon}$, triangle inequality tells us$$
\left| h _ { 1 } ( x ) - D _ { k l }v \right| \leq C | x | ^ { - 2 \varepsilon } , \quad | x | > 2 R _ { 1 }.
$$
Then Newton-Leibnitz formula tells us $$
\left| D ^ { j } w ( x ) \right| \leq C | x | ^ { 2 - j - 2 \varepsilon } , \quad | x | > 2 R _ { 1 } , \quad j = 0,1,2 .
$$
Higher regularity (when $m$ is larger than 3), the result follows by taking more derivatives.
\end{proof}

\begin{proof}[Proof of Theorem \ref{result2}]
  The theorem follows exactly as proved in \cite{BLZ}, which is based on Level Set Method and bootstrap argument we proved earlier. Since the proof is slightly different, we provide the details here again.

From the boundedness of Hessian matrix, we can verify that the conditions of Theorem \ref{LimitofHessian} still holds. Meaning that
there exist $A \in \mathcal{A}_3, \alpha>0, R_{2} \gg 1$ sufficiently large such that
\begin{equation*}
\left|D^{2} u(x)-A\right| \leq \frac{C}{|x|^{\alpha}}, \quad \forall|x| \geq R_{2}.
\end{equation*}
This enable us to apply Lemma \ref{roughestimate_2} which gives us the initial point of doing iteration as in Lemma \ref{bootstrap} to obtain a faster decay speed. We can always do this finite times till the condition of $0<\varepsilon<\dfrac{1}{2}$ in Lemma \ref{bootstrap} fails to hold. As in Lemma \ref{roughestimate_2}, we denote $w(x)=u(x)-\frac{1}{2} x^{T} A x$, where $A$ is the limit of $D^2u$ at infinity, which is provided by Theorems \ref{LimitofHessian} and \ref{HolderRegularity_ofHessian} in Section \ref{ConvergenceOfHessian}.

  Let $k_{0}$ be the positive integer such that $2^{k_{0}} \varepsilon<1$ and $2^{k_{0}+1} \varepsilon>1$ (we choose $\varepsilon$ smaller if
necessary to make both inequalities hold). Let $\varepsilon_{1}=2^{k_{0}} \varepsilon,$ clearly we have $1<2 \varepsilon_{1}<2$ .
Applying Lemma \ref{roughestimate_2} $k_{0}$ times we have
\begin{equation*}
\left|D^{k} w(x)\right| \leq C|x|^{2-\varepsilon_{1}-k}, \quad k=0, \ldots, m+1, \quad|x|>2 R_{1},
\end{equation*}
and
\begin{equation*}
\frac{\left|D^{m+1} w\left(x_{1}\right)-D^{m+1} w\left(x_{2}\right)\right|}{\left|x_{1}-x_{2}\right|^{\alpha}} \leq C\left|x_{1}\right|^{1-m-\varepsilon_{1}-\alpha},\left|x_{1}\right|>2 R_{1}, x_{2} \in B_{\left|x_{1}\right| / 2}\left(x_{1}\right).
\end{equation*}

Let $h_{1}$ and $f_{2}$ be the same as in Lemma \ref{bootstrap}. Then we have

\begin{equation*}
\left|f_{2}(x)\right| \leq C|x|^{1-m-2 \varepsilon_{1}}+C|x|^{-2-\zeta} \quad|x| \geq 2 R_{1},
\end{equation*}
and \begin{equation*}
\frac{\left|f_{2}\left(x_{1}\right)-f_{2}\left(x_{2}\right)\right|}{\left|x_{1}-x_{2}\right|^{\alpha}} \leq \frac{C}{\left|x_{1}\right|^{m-1+2 \varepsilon_{1}+\alpha}}+\frac{C}{|x|^{\zeta+2+\alpha}}, \quad\left|x_{1}\right| \geq 2 R_{1},  x_{2} \in B_{\left|x_{1}\right| / 2}\left(x_{1}\right).
\end{equation*}
Constructing $h_{2}$ as in the proof of Lemma \ref{bootstrap}, we have
\begin{equation*}
\left|D^{j} h_{2}(x)\right| \leq C|x|^{-2 \varepsilon_{1}-j}, \quad j=0,1,2,|x|>2 R_{1},
\end{equation*}
and
\begin{equation*}
\frac{\left|D^{2} h_{2}\left(x_{1}\right)-D^{2} h_{2}\left(x_{2}\right)\right|}{\left|x_{1}-x_{2}\right|^{\alpha}} \leq \frac{C}{\left|x_{1}\right|^{2+2 \varepsilon_{1}+\alpha}}, x_{2} \in B_{\frac{\left|x_{1}\right|}{2}}\left(x_{1}\right),\left|x_{1}\right|>2 R_{1}.
\end{equation*}
Similar to the proof of Lemma \ref{bootstrap}, it follows from the estimate above that
\begin{equation*}
\left|h_{1}(x)-h_{2}(x)\right| \leq C|x|^{2-n}, \quad|x|>2 R_{1}.
\end{equation*}
Since $2 \varepsilon_{1}>1$,
\begin{equation*}
\left|h_{1}(x)\right| \leq\left|h_{2}(x)\right|+C|x|^{2-n} \leq C|x|^{-1}.
\end{equation*}
By Theorem 4 of \cite{20}, for any $i=1,2,\cdots,n$, $ \partial_{i} w(x) \rightarrow c_{i}$ for some $c_{i} \in \mathbb{R}$ as $|x| \rightarrow \infty .$ Let $\beta \in \mathbb{R}^{n}$ be the limit
of $D w$ and $w_{1}(x)=w(x)-\beta \cdot x .$ The equation satisfied by $\partial_e w_{1}$ can be written as (for $e \in \mathbb{S}^{n-1} )$
\begin{equation*}
a_{i j} \partial_{i j}\left(\partial_{e} w_{1}\right)=\partial_{e} f_{v},\text{ where }a_{ij}=D_{M_{ij}}F_{\tau}(D^2u)\rightarrow D_{M_{ij}}F_{\tau}(A).
\end{equation*}
Write the equation into the perturbation of elliptic equation with constant coefficients, then from the estimate above we have
\begin{equation*}
D_{M_{ij}}F_{\tau}(A)\left(\partial_{e} w_{1}\right)=f_{3} :=\partial_{e} f_{v}-\left(D_{M_{ij}}F_{\tau}(D^2u)-D_{M_{ij}}F_{\tau}(A)\right) \partial_{i j e} w_{1}, \quad|x|>2 R_{1},
\end{equation*}
\begin{equation*}
\left|f_{3}(x)\right| \leq C\left(|x|^{-\zeta-1}+|x|^{-1-2 \varepsilon_{1}}\right) \leq C|x|^{-1-2 \varepsilon_{1}},|x|>2 R_{1},
\end{equation*}
and
\begin{equation*}
\frac{\left|f_{3}\left(x_{1}\right)-f_{3}\left(x_{2}\right)\right|}{\left|x_{1}-x_{2}\right|^{\alpha}} \leq C\left|x_{1}\right|^{-1-2 \varepsilon_{1}-\alpha},\left|x_{1}\right|>2 R_{1}, x_{2} \in B_{\left|x_{1}\right| / 2}\left(x_{1}\right).
\end{equation*}

Let $h_{4}$ solve $D_{M_{ij}}F_{\tau}(A) h_{4}=f_{3}$ and the construction of $h_{4}$ is similar to that of $h_{2}$ in Lemma \ref{bootstrap} . Then we have
\begin{equation*}
\left|D^{j} h_{4}(x)\right| \leq C|x|^{1-2 \varepsilon_{1}-j},|x|>2 R_{1}, j=0,1,2,
\end{equation*}
and
\begin{equation*}
\frac{\left|D^{2} h_{4}\left(x_{1}\right)-D^{2} h_{4}\left(x_{2}\right)\right|}{\left|x_{1}-x_{2}\right|^{\alpha}} \leq C\left|x_{1}\right|^{-1-2 \varepsilon_{1}-\alpha},\left|x_{1}\right|>2 R_{1}, x_{2} \in B_{\left|x_{1}\right| / 2}\left(x_{1}\right).
\end{equation*}
Since $\partial_{e} w_{1}-h_{4} \rightarrow 0$ at infinity, we have
\begin{equation*}
\left|\partial_{e} w_{1}(x)-h_{4}(x)\right| \leq C|x|^{2-n}, \quad|x|>R_{1}.
\end{equation*}
Therefore we have obtained $\left|Dw_{1}(x)\right| \leq C|x|^{1-2 \varepsilon_{1}}$ on $|x|>R_{1} .$ Using fundamental theorem
of calculus it tells us
\begin{equation*}
\left|w_{1}(x)\right| \leq C|x|^{2-2 \varepsilon_{1}}, \quad j=0,1, \quad|x|>R_{1}.
\end{equation*}
Lemma \ref{roughestimate_2} applied to $w_1$ gives
\begin{equation*}
\left|D^{j} w_{1}(x)\right| \leq C|x|^{2-j-2 \varepsilon_{1}}, \quad j=0 . ., m+1.
\end{equation*}

This provides us a finer estimate on $f_3$ i.e.
\begin{equation*}
\left|f_{3}(x)\right| \leq C|x|^{-\zeta-1}+C|x|^{-1-4 \varepsilon_{1}},|x|>2 R_{1},
\end{equation*}
and
\begin{equation*}
\frac{\left|f_{3}\left(x_{1}\right)-f_{3}\left(x_{2}\right)\right|}{\left|x_{1}-x_{2}\right|^{\alpha}} \leq C\left(\left|x_{1}\right|^{-\zeta-1-\alpha}+\left|x_{1}\right|^{-1-4 \varepsilon_{1}-\alpha}\right),\left|x_{1}\right|>2 R_{1}, x_{2} \in B_{\left|x_{1}\right| / 2}\left(x_{1}\right).
\end{equation*}
As a consequence, the new estimate of $h_4$ is
\begin{equation*}
\left|h_{4}(x)\right| \leq C\left(|x|^{1-\zeta}+|x|^{1-4 \varepsilon_{1}}\right), \quad|x|>2 R_{1},
\end{equation*}
and as always, triangle inequality tells us
\begin{equation*}
\left|D w_{1}(x)\right| \leq C\left(|x|^{2-n}+|x|^{1-4 \varepsilon_{1}}\right) \leq C|x|^{-1}, \quad|x|>2 R_{1}.
\end{equation*}
By Theorem 4 of \cite{20} again, there exists some constant $\gamma$ such that $w_{1} \rightarrow \gamma$ at infinity. Let
\begin{equation*}
w_{2}(x)=w_1(x)-\gamma=w(x)-\beta \cdot x-\gamma
\end{equation*}
Then we have $\left|w_{2}(x)\right| \leq C$ for $|x|>2 R_{1} .$ Lemma \ref{roughestimate_2} applied to $w_{2}$ gives
\begin{equation*}
\left|D^{k} w_{2}(x)\right| \leq C|x|^{-k}, \quad k=0, \ldots, m+1, \quad|x|>2 R_{1}.
\end{equation*}
The equation satisfied by $w_{2}$ can be written as
\begin{equation*}
F_{\tau}(A+D^2w_2(x))=f.
\end{equation*}
Taking the difference between this equation and $F_{\tau}(A)=f(\infty)$ we have
\begin{equation*}
\widetilde{a}_{i j} \partial_{i j} w_{2}=f-f(\infty), \quad|x|>2 R_{1},
\end{equation*}
where $$
\widetilde{a}_{ij}(x):=\int_0^1D_{M_{ij}}F_{\tau}(A+tD^2w_2(x))\d t.
$$
Note that the convergence of Hessian (Theorem \ref{LimitofHessian}) tells us $\bolangxian{a_{ij}}$ satisfy
\begin{equation*}
\left|D^{k}\left(\widetilde{a}_{i j}(x)-\widetilde{a}_{i j}(\infty)\right)\right| \leq C|x|^{-2-k}, \quad|x|>2 R_{1}, \quad k=0,1.
\end{equation*}
Write it into the perturbed situation of constant coefficients operator again
\begin{equation*}
\widetilde{a}_{ij}(\infty)D_{ij}w_{2}=f_{4} :=f_{v}-1-\left(\widetilde{a}_{i j}-\widetilde{a}_{ij}(\infty)\right) \partial_{i j} w_{2}, \quad|x|>2 R_{1},
\end{equation*}
and we have the following estimates
\begin{equation*}
\left|f_{4}(x)\right| \leq C\left(|x|^{-\zeta}+|x|^{-4}\right), \quad|x|>2 R_{1},
\end{equation*}
and
$$
\frac{\left|f_{4}\left(x_{1}\right)-f_{4}\left(x_{2}\right)\right|}{\left|x_{1}-x_{2}\right|^{\alpha}} \leq C\left(\left|x_{1}\right|^{-\zeta-\alpha}+\left|x_{1}\right|^{-4-\alpha}\right), \quad\left|x_{1}\right|>2 R_{1}, x_{2} \in B_{\left|x_{1}\right| / 2}\left(x_{1}\right).
$$
Let $h_{5}$ be defined similar to $h_{2} $, which makes $h_{5}$ solves $\widetilde{a}_{ij}(\infty) h_{5}=f_{4}$ in $\mathbb{R}^{n} \backslash B_{2 R_{1}}$ and satisfies
\begin{equation*}
\left|h_{5}(x)\right| \leq C\left(|x|^{2-\zeta}+|x|^{-2}\right).
\end{equation*}
As before we have
\begin{equation*}
\left|w_{2}(x)-h_{5}(x)\right| \leq C|x|^{2-n}, \quad|x|>2 R_{1},
\end{equation*}
which gives
\begin{equation}\label{2.30}
\left|w_{2}(x)\right| \leq C\left(|x|^{2-n}+|x|^{2-\zeta}+|x|^{-2}\right), \quad|x|>2 R_{1}.
\end{equation}
If $|x|^{-2}>|x|^{2-n}+|x|^{2-\zeta}$ we can apply the same argument as above finite times to remove
the $|x|^{-2}$ from (\ref{2.30}). Eventually by Lemma \ref{roughestimate_2} we have this result.
\end{proof}

\section{Interior Estimates}\label{InteriorEstimates}

From the proof of Theorem \ref{LimitofHessian}, we see that as long as (\ref{linear-of-X}) holds, then all the asymptotic behavior also holds. In this section, we will reduce the assumption on Hessian (\ref{boundedHessian}) into assumption on gradient (\ref{Condition-LinearGrowth}) or assumption on solution itself (\ref{Condition-QuadraticGrowth}).

This section is organized in the following order. First, we  prove that condition (\ref{Condition-LinearGrowth}) provides us the desired equivalence (\ref{linear-of-X}) immediately. Second, based on an important gradient estimate theorem by Y.Y.Li \cite{Estimate-YYLi}, we can furthermore reduce the condition (\ref{Condition-LinearGrowth}) by (\ref{Condition-QuadraticGrowth}). These two parts provide us Theorems \ref{lineargrowth_Result} and \ref{quadraticgrowth_Result} already. Eventually, by the compactness method developed by McGonagle-Song-Yuan \cite{CompactnessMethod}, we see that the boundeness of Hessian holds true under a weaker assumption of right hand side function than in Theorems \ref{lineargrowth_Result} and \ref{quadraticgrowth_Result}.

\begin{lemma}[Linear Growth]\label{lineargrowth}
Let $(\widetilde{x},\widetilde{u})$ be the function defined as in (\ref{LegendreTransform}), and $u$ satisfy condition (\ref{Condition-LinearGrowth}) for some constant $C_0$.  Then there exist $\delta=\delta(n,C_0,\tau),\ \displaystyle M=M(\tau,\max_{x\in\overline{B_1}}|Du(x)|)>0$ such that $$
  |(\bolangxian{x}-D\bolangxian{u}(\bolangxian{x}))|\geq \delta|\bolangxian{x}|,\quad\forall|\bolangxian{x}|\geq M.
  $$
\end{lemma}
\begin{proof}[Proof.]

Firstly, we prove that $$
|2bx|\geq \delta|Du(x)+(a+b)x|,\quad\forall |x|\geq  1.
$$
In fact,  by triangle inequality $$
|Du(x)+(a+b)x|
\leq |Du(x)|+(a+b)|x|\leq (C_0+a+b)(|x|+1),\quad\forall x\in\mathbb{R}^n.
$$
Thus by taking a sufficiently small $\delta>0$ such that $$
(C_0+a+b)\delta<\dfrac{2b}{100},\quad\text{hence}
\quad
2b|x|\geq 100\delta(C_0+a+b)|x|.
$$
Thus as long as $|x|>\dfrac{1}{99}$, we have $$
100\delta(C_0+a+b)|x|
\geq \delta(C_0+a+b)(|x|+1).
$$
Hence there exists a $\delta>0$ such that  $$
|2bx|\geq \delta|Du(x)+(a+b)x|,\quad\forall |x|\geq  1.
$$

Secondly, we prove the following result, then the result follows immediately,
\begin{equation}\label{temp-14}
\{x:|\widetilde{x}|=|D u(x)+(a+b) x| \geq M\} \subset\{x:|x| \geq 1\}.
\end{equation}
In fact, for any given $u\in C^4(\mathbb{R}^n)$, $|Du(x)|$ is bounded in $B_1$. Let's denote $$
\bolangxian{M}:=\max_{x\in\overline{B_1}}|Du(x)|<\infty,\quad\text{depending on $u$.}
$$
Then for any $|x|\leq 1$, $$
|Du(x)+(a+b)x|\leq
|Du(x)|+(a+b)|x|\leq \bolangxian{M}+a+b=:M.
$$
This tells us (\ref{temp-14}) through an argue by contradiction.
\end{proof}

Now we consider solution of (\ref{equation_NONConstant}) with $D^2u>(-a+b)I$ and we assume that $f$ satisfies conditions (\ref{0-orderCondition}) and (\ref{Holder-Condition}) for some $f(\infty)<0,\ \alpha\in(0,1),\ \beta>1,\ \gamma>0$. The following gradient estimate by Y.Y.Li \cite{Estimate-YYLi} plays an important role in our proof.
\begin{theorem}[Interior Gradient Estimate]\label{estimate_YYLi}
  Consider the following equation $$
f\left(\lambda_{1}, \ldots, \lambda_{n}\right)=\psi(x, u(x), Du(x)),
$$
where $\lambda_i$ is the eigenvalue of $D^2u(x)$.
Suppose $f$
\begin{enumerate}[(A)]
  \item is smooth and defined in an open convex cone $\Gamma \subset R^{n},$
which is different from $R^{n}$ , with vertex at the origin, containing the positive
cone $\left\{\lambda \in R^{n} : \text { each component } \lambda_{i}>0\right\}$
\item satisfies the following in $\Gamma $
$$
\frac{\partial f}{\partial \lambda_{i}}>0 \quad \forall i,\quad \text{and  $f$ is a concave function.}
$$
\item is invariant under interchange of any two $\lambda_{i}$s.
\item  there exist $c_{0}, c_{1}>0,$ which depend on $\psi_{0},$ such that,
for any $\lambda \in \Gamma$ with $f(\lambda) \geqslant \psi_{0}, \lambda_{i} \leqslant 0$ $$
f_{\lambda_{i}} \geqslant c_{0} \sum_{j \neq i} \frac{\partial f}{\partial \lambda_{j}}+c_{1}, \qquad i=1, \ldots, n.
$$
\item \label{condition} satisfies $$
\lim _{\lambda \rightarrow 0 \atop \lambda \in \Gamma} f(\lambda)>-\infty,
$$
\end{enumerate}
and we assume RHS term satisfies $\psi(x, u, v)$ is assumed to be
$C^{1},$ satisfying $$
\psi_{u} \geqslant 0, \quad|\psi| \leqslant \psi_{1}, \quad|\nabla \psi| \leqslant M, \quad \psi \geqslant \psi_{0},
$$ for some positive constant
$M, \psi_{1}, \psi_{0}$
and
$$
\lim _{\lambda \rightarrow \lambda_{0} \atop \lambda \in \Gamma} f(\lambda)<\psi_{0}, \quad \forall \lambda_{0} \in \partial \Gamma.
$$
If $u \in C^{3}\left(\overline{B_{1}}\right)$ is a solution of $$
f\left(\lambda_{1}, \ldots, \lambda_{n}\right)=\psi(x, u(x), Du(x)),\quad\text{in}\ B_1,
$$
with $\lambda=\lambda\left(S_{u}, P\right) \in \Gamma, u<0$ on $\overline{B_1}$ and $u(0)=-u_{0}$.
Then there exists $C=C\left(n, f, \psi_{1}, M, \psi_{0}, u_{0}\right)$ such
that $|\nabla u(0)| \leqslant C .$
\end{theorem}

In order to write our equation into a suitable form that satisfies condition (\ref{condition}) of Theorem \ref{estimate_YYLi}, we see from  proof of Theorem \ref{LimitofHessian} that there exists $\delta>0$ such that
\begin{equation}\label{ExistenceOfDelta}
\lambda_{i}\left(D^{2} u\right)>-a+b+\delta.
\end{equation}
Thus we translate by setting
$$
v:=u+\dfrac{a-b-\delta}{2}|x|^2,$$ such that
$\lambda_i(D^2v)=\lambda_i(D^2u)+(a-b)-\delta>0.$
Then $v$ satisfies
\begin{equation}\label{temp-23}
G_{\tau}(\lambda(D^2v)):=\prod_{i=1}^n\dfrac{\lambda_i+\delta}{\lambda_i+\delta+2b}=
\exp\{\frac{2b}{\sqrt{a^2+1}}
f(x)\}=:g(x),\quad\forall\ x\in\mathbb{R}^n.
\end{equation}

\begin{theorem}[Global Gradient Estimate]\label{gradientestimate}Let  $v$ be a smooth and convex solution of (\ref{temp-23}),
where $g$  satisfies conditions (\ref{0-orderCondition}) and (\ref{Holder-Condition}) for some $g(\infty)=\exp\{\dfrac{2b}{\sqrt{a^2+1}}f(\infty)\}>0,\ f(\infty)<0,\ \alpha\in(0,1),\ \beta>1,\ \gamma>0$.
Suppose in addition that there exists some constant $C_0$ such that (\ref{temp-9}) holds.
Then there for some $C>0$  such that (\ref{Condition-LinearGrowth}) holds.
\end{theorem}
\begin{proof}[Proof.]
For any sufficiently large $|x|\gg 1$, we consider the equation in $B_{\frac{|x|}{2}}(x)$ and let $$
\bolangxian{v}(y):=\dfrac{1}{(\frac{|x|}{2})^2}v(x+\dfrac{|x|}{2}y),\quad y\in B_1(0).
$$
Then we have $$
D^2\bolangxian{v}(y)=D^2v(x+\dfrac{|x|}{2}y),\quad\forall y\in B_1(0),
$$
hence it satisfies $$
G_{\tau}(D^2\bolangxian{v}(y))=g(x+\dfrac{|x|}{2}y)=:\psi(y),\quad\text{in } B_1(0).
$$
We   see that the right hand side term satisfies $$
|\psi|\leq \psi_1,\quad \psi\geq \psi_0,\quad\text{and}\quad
|\nabla\psi(y)|=\dfrac{|x|}{2}|\nabla g(x+\dfrac{|x|}{2}y)|\leq M,
$$
for some constants $\psi_1, \psi_0, M$  uniform to all $|x|\gg 1$ from the asymptotic behavior of $|\nabla g|$.

Also, we have $$
||\bolangxian{v}||_{L^{\infty}(B_1)}\leq \dfrac{1}{(\frac{|x|}{2})^2}
||v||_{L^{\infty}(B_{\frac{|x|}{2}}(x))}
\leq C_nC_0,
$$
which is also uniform with respect to $x\in\mathbb{R}^n$.

Now we apply the interior gradient estimate theorem i.e. Theorem \ref{estimate_YYLi} to see  that there exists a uniform (to $x$) constant $C$ such that $$
|D\bolangxian{v}(0)|\leq C,
$$
where C only relies on $n,f,\psi_1,M,\psi_0,C_0$. Hence the constant C is uniform to our choice of $|x|\gg 1$.

This is exactly telling us the gradient of $u$ has at most 1-order growth.  To be more precise, we have $$
|D\bolangxian{v}(0)|=\left|
\dfrac{2}{|x|}Dv(x)
\right|\leq C,\quad\text{uniform to}\quad|x|\gg 1.
$$
Hence there exists some constant $C$ such that condition (\ref{Condition-LinearGrowth}) holds true.
\end{proof}

Theorems \ref{lineargrowth_Result} and \ref{quadraticgrowth_Result} follow directly from Lemma \ref{lineargrowth} and Theorem \ref{gradientestimate}.

The scaling strategy used here is exactly the same as  reducing ``bounded Hessian'' to ``linear growth gradient'' condition.
Now we finish this section with proving that under the condition (\ref{Condition-LinearGrowth}), $|D^2u(x)|$ is bounded on entire $x\in\mathbb{R}^n$ using  compactness method as  in \cite{CompactnessMethod}.
To be more precise, we have the following result.
\begin{theorem}[Global Hessian Estimate]\label{mainestimate}Let $u$ be a smooth solution of (\ref{equation_NONConstant})
  with $$
  D^2u>(-a+b)I,
  $$
 where $f$  satisfies conditions (\ref{0-orderCondition}) and (\ref{Holder-Condition}) with $1$ replaced by some $f(\infty)<0,\ \alpha\in(0,1),\ \beta>1,\ \gamma>0$.
  Suppose that there exists some constant $C_0>0$ such that (\ref{Condition-LinearGrowth}) holds,
  then there exists $C=C(n,\alpha,C_0,f(\infty),\tau)$ such that (\ref{boundedHessian}) holds.
\end{theorem}
\begin{proof}[Proof.]
Based on Lemma \ref{lineargrowth}, we learn that under the condition that gradient $|Du(x)|$ has at most linear growth, there exist $\delta=\delta\left(n, C_{0}, \tau\right), M=M\left(n, C_{0}, \tau, u\right)>0$ such that
$$
|(\widetilde{x}-D \widetilde{u}(\widetilde{x}))| \geq \delta|\widetilde{x}|, \quad \forall|\widetilde{x}| \geq M.
$$

Now for sufficiently large $|x|\gg 1$, we consider in the ball $B_{\frac{|x|}{2}}(x)$, by setting $$
v(y):=\dfrac{4}{|x|^2}u(x+\dfrac{|x|}{2}y),\quad\forall y\in B_1(0),
$$
then $v$ satisfies $$
||Dv||_{C^0(\overline{B_1})}\leq C_0+1,
$$ and  $$
F_{\tau}(\lambda(D^2v))=\frac{2b}{\sqrt{a^2+1}}f(x+\dfrac{|x|}{2}y)=:g(y),\quad\forall y\in B_1(0).
$$
From our decay condition, we can easily prove that $$||g||_{C^{\alpha}(B_1)}\leq M,$$
uniformly.

Now we only need to prove that $$
|D^2v(0)|\leq C(M+1),
$$
for some constant M uniform to $x,v,g$. This is proved through argue by contradiction.

\textit{Step 1: Argue by Contradiction, $L^1$-convergence result}

Suppose the result doesn't hold, then there exist  sequences of smooth functions $\{u_k\}_{k=1}^{\infty}, \{g_k\}_{k=1}^{\infty}$ with $$D^2u_k>(-a+b)I,\quad\forall\ k=1,\cdots,$$
satisfying $$F_{\tau}(\lambda(D^2u_k))=g_k(y),\text{ in }B_1,$$
$$
||g_k||_{C^{\alpha}(B_1)}\leq M,
$$
$$
|D^2u_k(0)|>k\left(||Du_k||_{L^{\infty}(B_1)}+1\right),
$$
and having  uniform (to $k$) bounded $$||Du_k||_{L^{\infty}(B_1)}
\leq ||Du||_{L^{\infty}(B_1)}\leq C_0+1.
$$

Then using integral by parts, we obtain the following $W^{2,1}$ estimate.  Let $B_{1}^{m}$ denote the $m$ dimensional ball $B_{1}^{m}(0) \subset B_{1}=B_{1}^{n}(0) \subset \mathbb{R}^{n}$ for all $m=1, \cdots, n$ and $B^m_1(x):=B_1^m+x$ is the ball centered at $x$. Note that we only need to verify that every component of $D^2u_k$ belongs to $L^1$. Hence we use Fubini theorem and for any positive $n\times n$ matrix we have $$
||A||:=\sup_{|x|=1}|Ax|\leq \sum_{i=1}^n\lambda_i(A),
$$ it follows that
$$
\begin{array}{llllll}
  \displaystyle\int_{B_1}||D^2u_k+(a-b)I||\d y & = & \displaystyle \int_{B_{\sqrt{1-|x|^2}}^{n-m}(x)}\int_{B_r^m}||D^2u_k+(a-b)I||\d x \d r\\
  & \leq &\displaystyle C_n \int_{B_{\sqrt{1-|x|^2}}^{n-m}(x)}\int_{B_r^m}(\Delta u_k+C_n(a-b))\d x \d r\\
  & \leq & \displaystyle C_n(a-b)+\int_{B_1}\Delta u_k\d x.\\
\end{array}
$$
Integral by parts to the formula above and since the solutions $\{u_k\}$ are smooth up to boundary, hence $$
\begin{array}{llll}
  ||u_k||_{W^{2,1}(B_1)} & \leq & C_{n}(a-b)+\int_{\partial B_{1}} D u_{k} \cdot \vec{n} \mathrm{d} x\\
  &\leq & C_n((a-b)+||D u_k||_{L^{\infty}(B_1)})\leq C.\\
\end{array}
$$

Thus by the compact Sobolev Embedding $u_{k} \in W^{2,1}\left(B_{1}^{m}\right) \hookrightarrow W^{1,1}\left(B_{1}^{m}\right)$,
 for almost all $H^m-$section of $B_1$, meaning that for $H^{n-m}$-almost all $\left(x_{m+1}, \cdots, x_{n}\right)$ in $B_{1} \cap \mathbb{R}^{n-m}$, there exists a function $u_{\infty}\in W^{1,1}(B_1^m)$ such that $$
D u_{k} \rightarrow D u_{\infty} \text { in } L^{1}\left(B_{1}^{m}\right),\quad\text{as}\quad
k\rightarrow\infty.
$$

\textit{Step 2: Legendre-Lewy Transform}

Using Legendre transform as in (\ref{LegendreTransform}),  the equation $F_{\tau}(\lambda(D^2u_k))=g_k(y)$ becomes a uniformly elliptic Monge-Amp\`ere type equation (\ref{transformed}).   We assume without loss of generality that $D\overline{u}(0)=0$ by modifying a suitable linear function from u, which doesn't affect the equation at all.

As in Section \ref{Section-ConstantSituation}, $\widetilde{x}\in Du(B_1)+(a+b)B_1\supset B_b(0)$, $\widetilde{u_k}$ satisfies
\begin{equation}\label{temp-transformed}
\sum_{i=1}^{n} \ln \widetilde{\lambda_i}(D^2\widetilde{u_k})
=g_k\left(\frac{1}{2 b}(\widetilde{x}-D \widetilde{u_k}(\widetilde{x}))\right),\quad \bolangxian{x}\in B_b(0),
\end{equation}
and $$
\widetilde{\lambda}_{i}(D^2\widetilde{u_k})=1-\frac{2 b}{\lambda_{i}(D^2u_k)+a+b}\in (0,1).
$$
From (\ref{ExistenceOfDelta}), there exists $\delta>0$ such that $$ 0<\delta<\widetilde{\lambda}_{i}<1.
$$

Hence we can easily see that the equation after Legendre transform (\ref{temp-transformed}) is uniformly elliptic and concave.

\textit{Step 3: Uniform $C^{2,\alpha}$ Estimate}

By Lemma \ref{lineargrowth}, we see that under the condition (\ref{Condition-LinearGrowth}), we still have the important equivalence result (\ref{linear-of-X}). Due to $||g_k||_{C^{\alpha}(B_1)}\leq M$, we obtain that
\begin{equation*}
\left\|g_{k}\left(\frac{1}{2 b}(\widetilde{x}-D \widetilde{u_{k}}(\widetilde{x}))\right)\right\|_{C^{\alpha}(B_b)}\leq CM,
\end{equation*}
where $C=C(n,a,b,C_0)$.
By Schauder estimates of fully nonlinear elliptic equations, Theorem 8.1 and the Remark 3 after it in \cite{FullyNonlinear} tells us $$||\widetilde{u_k}||_{C^{2,\alpha}(B_{\frac{b}{2}}(0))}\leq C,$$
for some uniformly to $k$ constant $C$.

Thus by compact embedding for any $0<\varepsilon<\alpha$, $C^{2,\alpha}\hookrightarrow\hookrightarrow C^{2,\alpha-\varepsilon}$ , there exist a subsequence of $\widetilde{u_{k}},$ still denoted by $\widetilde{u_{k}},$ and $ g_{\infty}\in C^{2,\alpha-\varepsilon}(B_{\frac{b}{4}}(0))$, $\widetilde{u_{\infty}} \in C^{2,\alpha-\varepsilon}(B_{\frac{b}{4}}(0))$ such that  $$
\sum_{i=1}^{n} \ln \widetilde{\lambda}_{i}(D^2\widetilde{u_{\infty}}(x))
=g_{\infty}\left(\frac{1}{2 b}(\widetilde{x}-D \widetilde{u_{\infty}}(\widetilde{x}))\right),\quad \bolangxian{x}\in B_{\frac{b}{4}}(0).
$$

Due to $$
\left|D^{2} u_{k}(0)\right| \rightarrow \infty,\quad\text{as}\quad k\rightarrow\infty,
$$
hence we learned from Legendre transform that for some direction $\gamma\in\partial B_1$, $
D_{\gamma\gamma}\bolangxian{u_{\infty}}=0,
$
we may name this direction as ``$x_1$'' direction for simplicity.

\textit{Step 4: Contradicts Constant Rank Theorem}

The following constant rank theorem by Caffarelli-Guan-Ma \cite{constantrank} plays an important role in proving the result.

\begin{theorem}Let $\Psi \subset \mathbb{R}^{n}$ be an open symmetric domain, assume the operator $F(D^2u)=f(\lambda(D^2u))$ satisfies $f \in C^{2}(\Psi)$ symmetric and $$
f_{\lambda_{i}}(\lambda)=\frac{\partial f}{\partial \lambda_{i}}(\lambda)>0, \quad \forall i=1, \ldots, n, \quad \forall \lambda \in \Psi.
$$
Define $\widetilde{F}(A)=F\left(A^{-1}\right)$ whenever $A^{-1} \in \widetilde{\Psi}$, and we assume $\widetilde{F}$ is locally convex.
Assume $u$ is a $C^{3}$ convex solution of
the following equation in a domain $\Omega$ in $\mathbb{R}^{n} $
$$
F\left(D^2u(x)\right)=\varphi(x, u(x),D u(x)), \quad \forall x \in \Omega,
$$
for some $\varphi \in C^{1,1}\left(\Omega \times \mathbb{R} \times \mathbb{R}^{n}\right) .$ If $\varphi(x, u, p)$ is concave in $\Omega \times \mathbb{R}$ for any fixed
$p \in \mathbb{R}^{n},$ then the Hessian $D^2u$ has constant rank in $\Omega .$
\end{theorem}
Classical Hessian equations satisfies this condition, for example, all functions of $f(\lambda)=\sigma_k^{\frac{1}{k}}$ holds.  (Remark 1.7 of \cite{constantrank}).
Due to the rank of Hessian matrix is a constant inside, hence $D_{11} \widetilde{u_{\infty}}(y) \equiv 0$ in a neighborhood of 0. Hence on the hypersurface of $\left(x_{1}, y_{1}\right)$, we have $$
\left(D_{1} \widetilde{w_{\infty}}\left(y_{1}, y^{\prime}\right), y_{1}\right)=\left(c, y_{1}\right)=\left(x_{1}, D_{1} u_{\infty}\left(x_{1}, x^{\prime}\right)+ (a-b)x_{1}\right) \quad \text { near }(0,0).
$$
This is impossible as Step 1 tells us $\left(x_{1}, D_{1} u_{\infty}\left(x_{1}, x^{\prime}\right)+ (a-b)x_{1}\right)$ is an $L^1$ graph (for almost all $x'\in\mathbb{R}^{n-1}$).
\end{proof}

\section{Perturbed Results for $\tau=\frac{\pi}{4}$}\label{section-middle}
In this section, we prove the corresponding results for $\tau=\frac{\pi}{4}$ by similar strategy as in previous sections.
The proof is separated into the following two parts. First, as in Section \ref{ConvergenceOfHessian}, we prove the asymptotic behavior of $D^2u$ under additional assumption that Hessian $D^2u$ is bounded. Then the same argument as in Section \ref{ProofofResult1} and Section \ref{ProofOfResult2} shows us the desired result. Second, we reduce the assumption from bounded Hessian to linear growth of gradient and quadratic growth of $u$ as in Section \ref{InteriorEstimates}.

\subsection{Limit of Hessian}

Let $u\in C^4(\mathbb{R}^n)$ be a solution of (\ref{equation_middletau}) with $\tau=\frac{\pi}{4}$ and satisfy condition (\ref{boundedHessian_2}). We apply the same Legendre transform (\ref{LegendreTransform2}) as in Section \ref{Section-ConstantSituation}, then we see that $V(\widetilde{x})$ satisfies $$
\Delta V=-\frac{\sqrt{2}}{2}f(DV(\widetilde{x}))=:g(\widetilde{x}),\quad\text{in}\quad
\mathbb{R}^n.
$$
First, we prove the limit of Hessian at infinity result that corresponds to Theorem \ref{LimitofHessian}.
\begin{theorem}\label{LimitofHessian-midtau} Assume that $f$ satisfies
\begin{equation}\label{condition-f-midtau}
\limsup_{|x|\rightarrow\infty}|x|^{\varepsilon_0+k}|D^k(f(x)-f(\infty))|<\infty,\ \forall k=0,1,
\end{equation}
for some $\varepsilon_0>0$. Then there exist $A\in\mathcal{A}(f(\infty),-1), \alpha>0,\ C(n,f,\varepsilon_0)$ and $R_2(n,f,\varepsilon_0)$ such that
$$
|D^2u(x)-A|\leq \dfrac{C}{|x|^{\alpha}},\quad\forall |x|\geq R_2.
$$
\end{theorem}
\begin{proof}[Proof]

  By condition (\ref{boundedHessian_2}), we see that $$
  \widetilde{\lambda}_i(D^2\widetilde{u})=\frac{1}{\lambda_i(D^2u)+1}\geq \dfrac{1}{M+1}>0,\quad\forall i=1,2,\cdots,n.
  $$
  Also note that $f\in C^0(\mathbb{R}^n)$ has a finite lower bound. Due to $\widetilde{\lambda}_i>0$ holds for all $i=1,2,\cdots,n$, we also have the following $$
  \widetilde{\lambda}_i<\sum_{j=1}^n\widetilde{\lambda}_j=-\frac{\sqrt{2}}{2}f(DV(\widetilde{x}))
  \leq -\frac{\sqrt{2}}{2}\inf_{\mathbb{R}^n}f<\infty.
  $$
  Combine these two results, we see that \begin{equation}\label{temp-20}
  0<\dfrac{1}{M+1}\leq \widetilde{\lambda}_i\leq - \frac{\sqrt{2}}{2}\inf_{\mathbb{R}^n}f<\infty.
  \end{equation}
  As a consequence, there exists a constant $\delta=\delta(\tau,\inf_{\mathbb{R}^n}f)>0$ such that $D^2u>-I+\delta I$. More explicitly, we can pick any $\delta<-\frac{\sqrt{2}}{\inf_{\mathbb{R}^n}f}$, the result holds.

  Now we prove a similar equivalence result as formula (\ref{linear-of-X}). By the definition of Legendre transform (\ref{LegendreTransform2}), we see that for any $x^1, x^2\in \mathbb{R}^n$, $$
  |\widetilde{x^1}-\widetilde{x^2}|=|Du(x^1)-Du(x^2)+(x^1-x^2)|\geq \delta|x^1-x^2|,
  $$
  and similarly we have $$
  |\widetilde{x^1}-\widetilde{x^2}|=|Du(x^1)-Du(x^2)+(x^1-x^2)|\leq (M+1)|x^1-x^2|.
  $$
  Combine these two formulas above and we see that there exists some constant $C_0=C_0(\inf_{\mathbb{R}^n}f,\tau,M)$ such that (\ref{linear-of-X}) holds true again.

  Based on the equivalence (\ref{linear-of-X}), we see that $$
  \limsup_{|\widetilde{x}|\rightarrow\infty}g(\widetilde{x})=-\frac{\sqrt{2}}{2}f(\infty),
  $$
  and
  \begin{equation}\label{temp-17}
  \limsup_{|\widetilde{x}|\rightarrow\infty}|\widetilde{x}|^{\varepsilon_0+1}\cdot |Dg(\widetilde{x})|<\infty.
  \end{equation}

  Let $w(\widetilde{x}):=-\frac{\sqrt{2}f(\infty)}{4n}|x|^2$, which satisfies $\Delta w=-\frac{\sqrt{2}}{2}f(\infty)$ in $\mathbb{R}^n$. Then we consider the difference $\overline{v}(\widetilde{x}):=V-w$, satisfying $$
  \Delta\overline{v}=-\frac{\sqrt{2}}{2}(f(DV(\widetilde{x}))-f(\infty))=:\overline{g}(\widetilde{x}),
  \quad\text{in}\quad\mathbb{R}^n.
  $$

  Let $K(\widetilde{x}):=G_{\Delta}$ be the fundamental solution (or Green's function) of Laplacian, then as in Lemma \ref{existence}, we consider the auxiliary function $v\in C^2(\mathbb{R}^n)$ (by Schauder estimates and the regularity of $\overline{g}(\widetilde{x})$) $$
  v(\widetilde{x}):=K*\overline{g}(\widetilde{x})=\int_{\mathbb{R}^n}K(\widetilde{x}-\widetilde{y})g(\widetilde{y})
  \d\widetilde{y},
  $$
  which satisfies $\Delta v=\overline{g}$ in $\mathbb{R}^n$ and $v\rightarrow 0$ as $|\widetilde{x}|\rightarrow\infty$.

  As in the proof of Lemma \ref{existence}, we separate the domain into $E_1,E_2,E_3$ and see that $D^2v\rightarrow 0$ as $|\widetilde{x}|\rightarrow\infty$. More explicitly, there exists some constant $C$ such that
  \begin{equation}\label{temp-18}
  |D^2v|\leq C|x|^{1-(1+\varepsilon_0)}=C|x|^{-\varepsilon_0},\quad\forall x\in\mathbb{R}^n.
  \end{equation}

  Also note that $v\in C^2(\mathbb{R}^n)$, hence $D^2v$ is bounded on entire $\mathbb{R}^n$

  Hence we see that $$
  \Delta (V-w-v)=0,\quad\text{in}\quad\mathbb{R}^n,
  $$
  and due to $D^2V, D^2w, D^2v$ are all bounded on $\mathbb{R}^n$, hence so is $D^2(V-w-v)$.
  Applying Theorem \ref{ExteriorLiouville_C}, we see that the limit of $D^2(V-w-v)$ at infinity exists. (Or we can take twice derivatives to both sides and apply Liouville theorem). More explicitly, picking $k=2$ in formula (\ref{ConvergenceSpeed1}), there exists $$\widetilde{A}\in\{\widetilde{A}\in\mathtt{Sym}(n):\sum_{i=1}^n\lambda_i(\widetilde{A})=0,
  \dfrac{1}{M+1}I+\frac{\sqrt{2}f(\infty)}{2n}I\leq \widetilde{A}\leq \frac{\sqrt{2}}{2}\inf_{\mathbb{R}^n}fI+\frac{\sqrt{2}f(\infty)}{2n}I
  \},$$ such that  \begin{equation}\label{temp-19}
  \limsup_{|\widetilde{x}|\rightarrow \infty}|\widetilde{x}|^n\left|D^2(V-w-v)-\widetilde{A}\right|<\infty.
  \end{equation}
  Combine formula (\ref{temp-18}), (\ref{temp-19}) and the fact that $D^2w\equiv -\frac{\sqrt{2}f(\infty)}{2n}I$, it follows immediately that $\widetilde{B}:=\widetilde{A}-\frac{\sqrt{2}f(\infty)}{2n}I$ belongs to the set of $$\{\widetilde{B}:\sum_{i=1}^n\lambda_i(\widetilde{B})=-\frac{\sqrt{2}}{2}f(\infty),
  \dfrac{1}{M+1}I\leq\widetilde{B}\leq \dfrac{\sqrt{2}}{2}\inf_{\mathbb{R}^n}fI
  \},$$ such that $$
  \limsup_{|\widetilde{x}|\rightarrow\infty}|\widetilde{x}|^{\min\{\varepsilon_0,n\}}
  |D^2V-\widetilde{B}|<\infty.
  $$
  By the definition of Legendre transform (\ref{LegendreTransform2}) and the equivalence (\ref{linear-of-X}), we pick $\alpha:=\min\{n,\varepsilon_0\}$, the desired result follows immediately.

\end{proof}

\begin{proof}[Proof of Theorems \ref{result3} and  \ref{result4}]
  As in the calculus of Theorem \ref{LimitofHessian-midtau}, formula (\ref{temp-20}) and the Legendre transform (\ref{LegendreTransform2}), $$
  \dfrac{\sqrt{2}}{\inf_{\mathbb{R}^n}f}\leq \lambda_i(D^2u)\leq M.
  $$

  By direct calculus, we have $$
  \dfrac{\partial}{\partial\lambda_i}\left(\sum_{i=1}^n\dfrac{1}{\lambda_i+1}\right)
  =-\dfrac{1}{(1+\lambda_i)^2}\in (-\frac{1}{(\frac{\sqrt{2}}{\inf_{\mathbb{R}^n}f}+1)^2},-\dfrac{1}{(1+M)^2}),
  $$
  meaning that the equation (\ref{equation_middletau}) is uniformly elliptic. Also, we have $$
  \dfrac{\partial^2}{\partial\lambda_i^2}\left(\sum_{i=1}^n\dfrac{1}{\lambda_i+1}\right)
  =\dfrac{2}{(1+\lambda_i)^3}>0,
  $$
  hence the equation is also convex. Consider $$
  -F_{\tau}(D^2u)=-f(x),\quad\text{in}\quad\mathbb{R}^n,
  $$
  then apply the same argument as in the proof of Theorems \ref{result1} and  \ref{result2}, similar results also hold true.
\end{proof}

\subsection{Interior Estimates}

In this part, we reduce the strong assumption (\ref{boundedHessian_2}) into weaker ones. Similar to the proof of Lemma \ref{lineargrowth}, it holds true under this situation as well.

Now we consider solution of (\ref{equation_middletau}) with $\tau=\frac{\pi}{4},\ D^2u>-I$ and we assume that $f$ satisfies condition (\ref{condition-f-midtau}) for some $\varepsilon_0>0,\ f(\infty)<0$.
The proof is an explicit copy, we only used the ellipticity and concavity of equation. Hence Theorems \ref{lineargrowth_Result_2} and  \ref{quadraticgrowth_Result_2} follow immediately.

As in the proof of Theorem \ref{LimitofHessian-midtau} or Lemma \ref{ExistenceOfDelta}, we see that there exists a $\delta=\delta(\tau,\inf_{\mathbb{R}^n}f)>0$ such that $$
\lambda_i(D^2u)>-1+\delta.
$$
Then we set $$
v:=u+\dfrac{1-\delta}{2}|x|^2,\quad\text{then}\quad
\lambda_i(D^2v)=\lambda_i(D^2u)+1-\delta>0.
$$
And we see that $v$ satisfies $$
\sum_{i=1}^n\dfrac{1}{\lambda_i(D^2v)+\delta}=-\frac{\sqrt{2}}{2}f(x),\quad\text{in}\quad
\mathbb{R}^n.
$$
By verifying the conditions in Theorem \ref{estimate_YYLi}, we do scaling as in the proof of Theorem \ref{gradientestimate},
Theorem \ref{lineargrowth_Result_2} follows from
similar argument as in Theorem \ref{lineargrowth_Result}.

Note that the equation after Legendre-Lewy transform becomes Laplacian operator, standard Schauder estimates still hold true. The estimates to prove Theorem \ref{quadraticgrowth_Result_2} follows by similar compactness argument as in Theorem \ref{quadraticgrowth_Result}.






\begin{thebibliography}{10}

\bibitem{Adams}
Robert~A. Adams and John J.~F. Fournier.
\newblock {\em Sobolev spaces}, volume 140 of {\em Pure and Applied Mathematics
  (Amsterdam)}.
\newblock Elsevier/Academic Press, Amsterdam, second edition, 2003.

\bibitem{BCGJ}
Jiguang Bao, Jingyi Chen, Bo~Guan, and Min Ji.
\newblock Liouville property and regularity of a {H}essian quotient equation.
\newblock {\em Amer. J. Math.}, 125(2):301--316, 2003.

\bibitem{BLZ}
Jiguang Bao, Haigang Li, and Lei Zhang.
\newblock Monge-{A}mp\`ere equation on exterior domains.
\newblock {\em Calc. Var. Partial Differential Equations}, 52(1-2):39--63,
  2015.

\bibitem{Peroidic_MA}
L.~Caffarelli and Yan~Yan Li.
\newblock A {L}iouville theorem for solutions of the {M}onge-{A}mp\`ere
  equation with periodic data.
\newblock {\em Ann. Inst. H. Poincar\'{e} Anal. Non Lin\'{e}aire},
  21(1):97--120, 2004.

\bibitem{CL}
L.~Caffarelli and Yanyan Li.
\newblock An extension to a theorem of {J}\"{o}rgens, {C}alabi, and
  {P}ogorelov.
\newblock {\em Comm. Pure Appl. Math.}, 56(5):549--583, 2003.

\bibitem{CNS}
L.~Caffarelli, L.~Nirenberg, and J.~Spruck.
\newblock The {D}irichlet problem for nonlinear second-order elliptic
  equations. {III}. {F}unctions of the eigenvalues of the {H}essian.
\newblock {\em Acta Math.}, 155(3-4):261--301, 1985.

\bibitem{constantrank}
Luis Caffarelli, Pengfei Guan, and Xi-Nan Ma.
\newblock A constant rank theorem for solutions of fully nonlinear elliptic
  equations.
\newblock {\em Comm. Pure Appl. Math.}, 60(12):1769--1791, 2007.

\bibitem{7}
Luis~A. Caffarelli.
\newblock Interior {$W^{2,p}$} estimates for solutions of the
  {M}onge-{A}mp\`ere equation.
\newblock {\em Ann. of Math. (2)}, 131(1):135--150, 1990.

\bibitem{FullyNonlinear}
Luis~A. Caffarelli and Xavier Cabr\'{e}.
\newblock {\em Fully nonlinear elliptic equations}, volume~43 of {\em American
  Mathematical Society Colloquium Publications}.
\newblock American Mathematical Society, Providence, RI, 1995.

\bibitem{Calabi}
Eugenio Calabi.
\newblock Improper affine hyperspheres of convex type and a generalization of a
  theorem by {K}. {J}\"{o}rgens.
\newblock {\em Michigan Math. J.}, 5:105--126, 1958.

\bibitem{2-Hessian}
Li~Chen and Ni~Xiang.
\newblock Rigidity theorems for the entire solutions of 2-{H}essian equation.
\newblock {\em J. Differential Equations}, 267(9):5202--5219, 2019.

\bibitem{ChengandYau}
Shiu~Yuen Cheng and Shing-Tung Yau.
\newblock Complete affine hypersurfaces. {I}. {T}he completeness of affine
  metrics.
\newblock {\em Comm. Pure Appl. Math.}, 39(6):839--866, 1986.

\bibitem{11}
Miguel de~Guzm\'{a}n.
\newblock {\em Differentiation of integrals in {$R^{n}$}}.
\newblock Lecture Notes in Mathematics, Vol. 481. Springer-Verlag, Berlin-New
  York, 1975.
\newblock With appendices by Antonio C\'{o}rdoba, and Robert Fefferman, and two
  by Roberto Moriy\'{o}n.

\bibitem{Quotient-quadratic-Complex}
S\l~awomir Dinew and S\l~awomir Ko\l~odziej.
\newblock Liouville and {C}alabi-{Y}au type theorems for complex {H}essian
  equations.
\newblock {\em Amer. J. Math.}, 139(2):403--415, 2017.

\bibitem{FMM99}
L.~Ferrer, A.~Mart\'{\i}nez, and F.~Mil\'{a}n.
\newblock An extension of a theorem by {K}. {J}\"{o}rgens and a maximum
  principle at infinity for parabolic affine spheres.
\newblock {\em Math. Z.}, 230(3):471--486, 1999.

\bibitem{Flanders}
Harley Flanders.
\newblock On certain functions with positive definite {H}essian.
\newblock {\em Ann. of Math. (2)}, 71:153--156, 1960.

\bibitem{20}
D.~Gilbarg and James Serrin.
\newblock On isolated singularities of solutions of second order elliptic
  differential equations.
\newblock {\em J. Analyse Math.}, 4:309--340, 1955/56.

\bibitem{GT}
David Gilbarg and Neil~S. Trudinger.
\newblock {\em Elliptic partial differential equations of second order}, volume
  224 of {\em Grundlehren der Mathematischen Wissenschaften [Fundamental
  Principles of Mathematical Sciences]}.
\newblock Springer-Verlag, Berlin, second edition, 1983.

\bibitem{EntireParabolicMA}
Cristian~E. Guti\'{e}rrez and Qingbo Huang.
\newblock A generalization of a theorem by {C}alabi to the parabolic
  {M}onge-{A}mp\`ere equation.
\newblock {\em Indiana Univ. Math. J.}, 47(4):1459--1480, 1998.

\bibitem{huang2019entire}
Rongli Huang, Qianzhong Ou, and Wenlong Wang.
\newblock On the entire self-shrinking solutions to lagrangian mean curvature
  flow ii, 2019.

\bibitem{Equivalence2}
H.~Hueber and M.~Sieveking.
\newblock Uniform bounds for quotients of {G}reen functions on
  {$C^{1,1}$}-domains.
\newblock {\em Ann. Inst. Fourier (Grenoble)}, 32(1):vi, 105--117, 1982.

\bibitem{25}
Huai-Yu Jian and Xu-Jia Wang.
\newblock Continuity estimates for the {M}onge-{A}mp\`ere equation.
\newblock {\em SIAM J. Math. Anal.}, 39(2):608--626, 2007.

\bibitem{Jorgens}
Konrad J\"{o}rgens.
\newblock \"{U}ber die {L}\"{o}sungen der {D}ifferentialgleichung {$rt-s^2=1$}.
\newblock {\em Math. Ann.}, 127:130--134, 1954.

\bibitem{JostandXin}
J.~Jost and Y.~L. Xin.
\newblock Some aspects of the global geometry of entire space-like
  submanifolds.
\newblock volume~40, pages 233--245. 2001.
\newblock Dedicated to Shiing-Shen Chern on his 90th birthday.

\bibitem{Quotient-quadratic-Double}
An-Min {Li} and Li~{Sheng}.
\newblock {A Liouville Theorem on the PDE $\det(f_{i\bar j})=1$}.
\newblock {\em arXiv e-prints}, page arXiv:1809.00824, Sep 2018.

\bibitem{ExteriorLiouville}
Dongsheng {Li}, Zhisu {Li}, and Yu~{Yuan}.
\newblock {A Bernstein problem for special Lagrangian equations in exterior
  domains}.
\newblock {\em arXiv e-prints}, page arXiv:1709.04727, Sep 2017.

\bibitem{Quotient-quadratic-LRW}
Ming Li, Changyu Ren, and Zhizhang Wang.
\newblock An interior estimate for convex solutions and a rigidity theorem.
\newblock {\em J. Funct. Anal.}, 270(7):2691--2714, 2016.

\bibitem{Estimate-YYLi}
Yan~Yan Li.
\newblock Interior gradient estimates for solutions of certain fully nonlinear
  elliptic equations.
\newblock {\em J. Differential Equations}, 90(1):172--185, 1991.

\bibitem{CompactnessMethod}
Matt McGonagle, Chong Song, and Yu~Yuan.
\newblock Hessian estimates for convex solutions to quadratic {H}essian
  equation.
\newblock {\em Ann. Inst. H. Poincar\'{e} Anal. Non Lin\'{e}aire},
  36(2):451--454, 2019.

\bibitem{parabolicSigmaK_constant}
Saori Nakamori and Kazuhiro Takimoto.
\newblock A {B}ernstein type theorem for parabolic {$k$}-{H}essian equations.
\newblock {\em Nonlinear Anal.}, 117:211--220, 2015.

\bibitem{Quotient-Hessian}
Tu~A. Nguyen and Yu~Yuan.
\newblock A priori estimates for {L}agrangian mean curvature flows.
\newblock {\em Int. Math. Res. Not. IMRN}, (19):4376--4383, 2011.

\bibitem{Equivalence}
Yehuda Pinchover.
\newblock On the equivalence of {G}reen functions of second order elliptic
  equations in {${\bf R}^n$}.
\newblock {\em Differential Integral Equations}, 5(3):481--493, 1992.

\bibitem{Pogorelov}
A.~V. Pogorelov.
\newblock On the improper convex affine hyperspheres.
\newblock {\em Geometriae Dedicata}, 1(1):33--46, 1972.

\bibitem{Peroidic_MA2}
Eduardo~V. Teixeira and Lei Zhang.
\newblock Global {M}onge-{A}mp\'{e}re equation with asymptotically periodic
  data.
\newblock {\em Indiana Univ. Math. J.}, 65(2):399--422, 2016.

\bibitem{Quotient-quadratic-what}
Mao-Pei Tsui and Mu-Tao Wang.
\newblock A {B}ernstein type result for special {L}agrangian submanifolds.
\newblock {\em Math. Res. Lett.}, 9(4):529--535, 2002.

\bibitem{ParabolicMA_constant1}
Bo~Wang and Jiguang Bao.
\newblock Asymptotic behavior on a kind of parabolic {M}onge-{A}mp\`ere
  equation.
\newblock {\em J. Differential Equations}, 259(1):344--370, 2015.

\bibitem{Wang.Chong-paper}
Chong Wang, Rongli Huang, and Jiguang Bao.
\newblock On the second boundary value problem for lagrangian mean curvature
  equation, 2018.

\bibitem{Warren}
Micah Warren.
\newblock Calibrations associated to {M}onge-{A}mp\`ere equations.
\newblock {\em Trans. Amer. Math. Soc.}, 362(8):3947--3962, 2010.

\bibitem{Nonpolynomial}
Micah Warren.
\newblock Nonpolynomial entire solutions to {$\sigma_k$} equations.
\newblock {\em Comm. Partial Differential Equations}, 41(5):848--853, 2016.

\bibitem{WYSpecialLagrangian}
Micah Warren and Yu~Yuan.
\newblock A {L}iouville type theorem for special {L}agrangian equations with
  constraints.
\newblock {\em Comm. Partial Differential Equations}, 33(4-6):922--932, 2008.

\bibitem{parabolicMA_constant3}
Jingang Xiong and Jiguang Bao.
\newblock On {J}\"{o}rgens, {C}alabi, and {P}ogorelov type theorem and isolated
  singularities of parabolic {M}onge-{A}mp\`ere equations.
\newblock {\em J. Differential Equations}, 250(1):367--385, 2011.

\bibitem{Yu.Yuan1}
Yu~Yuan.
\newblock A {B}ernstein problem for special {L}agrangian equations.
\newblock {\em Invent. Math.}, 150(1):117--125, 2002.

\bibitem{Yu.Yuan2}
Yu~Yuan.
\newblock Global solutions to special {L}agrangian equations.
\newblock {\em Proc. Amer. Math. Soc.}, 134(5):1355--1358, 2006.

\bibitem{Peroidic_parabolicMA}
Wei Zhang and Jiguang Bao.
\newblock A {C}alabi theorem for solutions to the parabolic {M}onge-{A}mp\`ere
  equation with periodic data.
\newblock {\em Ann. Inst. H. Poincar\'{e} Anal. Non Lin\'{e}aire},
  35(5):1143--1173, 2018.

\bibitem{ParabolicMA_Constant2}
Wei Zhang, Jiguang Bao, and Bo~Wang.
\newblock An extension of {J}\"{o}rgens-{C}alabi-{P}ogorelov theorem to
  parabolic {M}onge-{A}mp\`ere equation.
\newblock {\em Calc. Var. Partial Differential Equations}, 57(3):Art. 90, 36,
  2018.

\bibitem{Ziemer}
William~P. Ziemer.
\newblock {\em Weakly differentiable functions}, volume 120 of {\em Graduate
  Texts in Mathematics}.
\newblock Springer-Verlag, New York, 1989.
\newblock Sobolev spaces and functions of bounded variation.

\end{thebibliography}


\small
\end{document}